\documentclass[11pt]{amsart}

\usepackage[T1]{fontenc}
\usepackage[utf8]{inputenc}

\usepackage[margin=1in]{geometry}

\usepackage{csquotes}
\usepackage[USenglish]{babel}

\usepackage{lmodern}
\usepackage{mathtools,amsfonts,amsthm,mathrsfs,amssymb}
\mathtoolsset{showonlyrefs}

\usepackage{textcomp}

\usepackage[
    backend=biber,
    style=numeric,
    sorting=nyt,
    sortcites=true,
    abbreviate=true,
    giveninits=true,
    useprefix=false,
    isbn=false,
    terseinits=true,
    maxbibnames=99
]{biblatex}
\DeclareFieldFormat{pages}{#1}
\usepackage[shortlabels]{enumitem}
\usepackage[dvipsnames]{xcolor}

\PassOptionsToPackage{
    pdfusetitle,
    colorlinks=true,
    pdfstartview=FitV,
    linkcolor=Blue,
    citecolor=Green,
    urlcolor=WildStrawberry,
    linktoc=page
}{hyperref}
\usepackage{bookmark}
\usepackage[capitalise,noabbrev]{cleveref}
\usepackage[all]{hypcap} 
\usepackage{microtype}

\newtheorem{theorem}{Theorem}
\newtheorem{lemma}[theorem]{Lemma}
\newtheorem{proposition}[theorem]{Proposition}

\theoremstyle{definition}

\newcommand{\bI}{\mathbf I}
\newcommand{\cI}{\mathcal I}
\newcommand{\EE}{\mathbb E}
\newcommand{\PP}{\mathbb P}
\newcommand{\NN}{\mathbb N}
\newcommand{\RR}{\mathbb R}
\newcommand{\RRpos}{\RR_{>0}}
\newcommand{\RRnn}{\RR_{\ge0}}
\newcommand{\eps}{\varepsilon}
\DeclareMathOperator{\mad}{mad}
\newcommand*{\dkl}[2]{D(#1\|#2)}

\numberwithin{equation}{section}

\title{Random independent sets and local sparsity}
\author[E.\ Davies]{Ewan Davies}
\address{Department of Computer Science, Colorado State University, Fort Collins CO, USA}
\email{research@ewandavies.org}
\date{September 3, 2026}

\begin{document}

\begin{abstract}
We analyze random constructions of independent sets in locally sparse graphs, specifically graphs with bounded maximum average degree in neighborhoods or with fractionally $r$-colorable neighborhoods.
Specializing our methods to finding large independent sets and low-weight fractional colorings, we focus on optimizing for marginals, but we also derive results that find many independent sets (i.e.\ give lower bounds on the independence polynomial) by optimizing for entropy.
Our main results generalize the local Shearer bound of Martinsson and Steiner for triangle-free graphs to graphs with few triangles and to graphs with fractionally $r$-colorable neighborhoods, in the latter case improving upon a result of Dhawan.
We also extend independence polynomial bounds obtained via induction to such graphs, improving upon known bounds obtained by local occupancy by relaxing the necessary hypotheses from a maximum degree condition to an average degree condition.
\end{abstract}

\maketitle

\section{Introduction}
\label{sec:introduction}

For a finite simple graph $G=(V,E)$, let $N_G(v)$ be the neighborhood of a vertex $v\in V$, set $N_G[v]=N_G(v)\cup\{v\}$, and let $d(v)=|N_G(v)|$ be the degree of $v$.
Where possible, we drop the subscripts for brevity.
We write $\cI(G)$ for the set of independent sets of $G$, including $\emptyset$.
For a fugacity parameter $\lambda$, the (univariate) independence polynomial of \(G\) is
\[
 Z_G(\lambda)=\sum_{I\in \cI(G)}\lambda^{|I|}.
\]
In the statistical physics literature, $Z_G(\lambda)$ is the partition function of the hard-core model and $\log Z_G(\lambda)$ is its free energy. 
The hard-core model is the distribution on $\cI(G)$ where each independent set $I$ is chosen with probability proportional to $\lambda^{|I|}$ (for $\lambda>0$):
\[ \pi_{G,\lambda}(I) = \frac{\lambda^{|I|}}{Z_G(\lambda)}. \]
Note that $Z_G(1)=|\cI(G)|$ counts the number of independent sets in \(G\).

The probabilistic method, especially in the form of analyzing an algorithm, has long been a powerful tool for finding independent sets in graphs. For examples of algorithms and probabilistic yielding lower bounds on the independence number, see e.g.~\cite{Gri83,AKS81,Wei81}
An elegant proof of the Caro--Wei theorem~\cite{Car79,Wei81} popularized by Alon and Spencer~\cite{AS16} can be seen as analyzing the greedy algorithm under a uniform random vertex ordering.
Shearer's theorem giving what remains the best-known lower bound on the independence number of a triangle-free graph of given average degree~\cite{She83} also has this flavor. 
More recently, Dhawan analyzed a variation on a greedy algorithm to find independent sets in graphs with (fractionally) $r$-colorable neighborhoods~\cite{Dhawan26Entropy}. 
A modern development of these ideas seeks not just a large independent set, but a distribution on independent sets with a given lower bound on the marginal at each vertex. This program was initiated by Kelly and Postle~\cite{KP24} and given the name \emph{fractional coloring with local demands}, and in basic settings such as the Caro--Wei theorem the elegant idea is implicit in prior work.
For our purposes, we can specialize to a version of local demands such that for some function $g : [0,\infty)\to (0,1]$, we seek a distribution on independent sets such that the marginal at each vertex $v$ is at least $g(d(v))$, though we also consider this concept for a weighted notion of degree.

Another line of research seeks bounds on the independence polynomial of a graph, which counts independent sets according to their size. Most relevant to this work are lower bounds obtained via local occupancy~\cite{DJPR17b,DKPS20Structure}, which involves an analysis of the marginals of the hard-core model, and via induction~\cite{SSSZ19,LS26,BHK25}. 
The univariate lower bound of Sah et al.~\cite{SSSZ19} and Lee--Seo's multivariate extension~\cite{LS26} give the independence polynomial analogue of the Caro--Wei theorem in the sense that the bound is tight for a disjoint union of cliques (and certain parameter values in the multivariate case). 
It is less clear whether there is a canonical independence polynomial bound corresponding to Shearer's independence number bound for triangle-free graphs of given average degree~\cite{She83}, or, more in line with the Caro--Wei theorem, degree-sequence versions thereof~\cite{She91,LRZ02}. 
The work of Buys, van den Heuvel, and Kang~\cite{BHK25} provides a compelling analogue of Shearer's theorem~\cite{She83}, but degree-sequence versions are developed for the first time in this work.
It is also important to mention the upper bound on the independence polynomial of regular graphs due to Kahn and Zhao~\cite{Kah01,Zha10} as influential in this line of work.

Fractional coloring with local demands and lower bounds on the independence polynomial can be connected by the Gibbs variational principle, which states that the free energy $\log Z_G(\lambda)$ is the maximum of the sum of the expected energy and the entropy over all distributions on $\cI(G)$. Write $\mathcal P(\cI(G))$ for the set of probability distributions on $\cI(G)$, $\bI$ for a random independent set, and $\EE_\mu|\bI|$ for the expected size of the random independent set under the distribution $\mu$. With $H(\mu)$ denoting Shannon entropy (in natural units, not binary) and $\dkl{\cdot}{\cdot}$ the Kullback--Leibler divergence, we have for any distribution $\mu\in\mathcal P(\cI(G))$ and any $\lambda>0$
\begin{equation}
  \log Z_G(\lambda) = \EE_\mu|\bI|\log\lambda + H(\mu) + \dkl{\mu}{\pi_{G,\lambda}}.
\end{equation}
The proof, which we omit, is a simple matter of expanding definitions (with the caveat that here and throughout we use natural logarithms). 
Since the Kullback--Leibler divergence is nonnegative, this gives a lower bound on the free energy in terms of the marginals and entropy of any distribution on independent sets:
\begin{equation}
  \log Z_G(\lambda)
  =\max_{\mu\in\mathcal P(\cI(G))}
  \bigl\{
    \EE_\mu|\bI|\log\lambda
    +H(\mu)
  \bigr\}.
\end{equation}
We can write equality above due to the choice $\mu=\pi_{G,\lambda}$.
Thus, an arbitrary $\mu\in\mathcal P(\cI(G))$ gives an affine lower bound on the free energy in terms of $\log\lambda$.
The larger the marginals of $\mu$, the larger the slope, but one may prefer to optimize for the contribution from the entropy instead. 
A simpler argument along the same lines gives that when $G$ contains an independent set of size $k$ we have $\log Z_G(\lambda)\ge k\log(1+\lambda)$ for any $\lambda>0$.
While we avoid explicitly casting our methods in a common framework as the analysis of a random process constructing an independent set, the connection is clear and motivates our approach.

We give novel results for fractional coloring with local demands and for lower bounds on the independence polynomial in two settings: graphs with few triangles and graphs with fractionally $r$-colorable neighborhoods. Here the energy-entropy tradeoff is important: our direct independence polynomial bounds are significantly stronger than the elementary bounds obtained by summing over all subsets of the large independent sets supplied by our local demand results.
This detail is discussed in \cref{app:details}.
As is standard in prior works, we give our proofs by induction on the number of vertices. 
We discuss an intriguing connection with random processes on $\cI(G)$ in \cref{app:random_processes}.
A potential advantage of reformulating the proofs through a random process is the clear route to efficient randomized sampling algorithms for the distributions we guarantee with local demand theorems. We avoid the details as they are tedious, but find the computational efficiency perspective compelling.

\subsection{Fractional coloring with local demands}

For fractional coloring, our methods seem to necessitate handling positive vertex weights even if one is initially interested in unweighted graphs.
The important notion is a weighted generalization of vertex degree using positive vertex weights $w\in \RRpos^{V}$.
For $S\subseteq V$, write $w(S)=\sum_{x\in S}w(x)$ for the total weight of vertices in $S$, and define the weighted degree of a vertex $v\in V$ by
\[
 d_w(v)=\frac{w(N(v))}{w(v)}.
\]
We typically work with vertex weights in a scale-free manner so that our results are unchanged under a scaling of $w$, so that constant weight functions recover the usual unweighted case.
For a probability distribution $\mu$ on $\cI(G)$ and random $\bI\sim \mu$, write $\mu(v)=\PP_{\mu}(v\in\bI)$ for the marginal of $v$.
For a function $q\in [0,1]^{V}$, we say that $\mu$ meets local demands \(q\) if $\mu(v)\ge q(v)$ for every $v\in V$.
Equivalently, a demand function $q$ is met by some distribution if and only if $q$ (interpreted as a vector) lies in the independence polytope of \(G\).
See the seminal work of Kelly and Postle~\cite{KP24} for a thorough treatment of the notion of local demands for fractional coloring.
In this work we specialize to functions $q$ that depend only on the weighted degree of a vertex, i.e., $q(v)=g(d_w(v))$ for some function $g:[0,\infty)\to(0,1]$ and weight function $w\in \RRpos^V$, and we will refer to $g$ as a demand function.

Since the influential work of Kelly and Postle~\cite{KP24,KP24a}, there has been substantial interest in fractional colorings with local demands.
Martinsson and Steiner's local Shearer bound~\cite{MS25} and Dhawan's entropy-based construction~\cite{Dhawan26Entropy} give two different analyses.
Further work motivated by this topic includes a precise analysis of the fractional chromatic number of high-girth $d$-degenerate graphs~\cite{ADN26}, and local demands results for $K_r$-free graphs~\cite{DDGMNSS26+}.
We continue in this line by extending the local Shearer bound to graphs with few triangles and to graphs with fractionally \(r\)-colorable neighborhoods.
The resulting bounds smoothly generalize the exact form of the local Shearer bound for triangle-free graphs. In the case of fractionally colorable neighborhoods, we improve the leading asymptotic order by a constant factor compared to Dhawan's result.
We give a precise comparison after \cref{thm:LDRColNbdhs}.
The sparse neighborhoods setting has received substantial attention in the context of standard (and list and correspondence) colorings too~\cite{AKS99,DJKP21Occupancy,DKPS20Structure,AIS19a,Vu02,PH21a}.

Our results generalize Shearer's bound on the independence number of a triangle-free graph of given average degree~\cite{She83} and Li--Rousseau--Zang's generalization of Shearer's bound to the setting of graphs with bounded maximum average degree in neighborhoods~\cite{LRZ01} (and the strengthening to a degree-sequence version of the bound~\cite{LRZ02}).
Our proofs follow Martinsson and Steiner's induction~\cite{MS25} and combine it with Dhawan's approach~\cite{Dhawan26Entropy} in the case of graphs with fractionally $r$-colorable neighborhoods.
The main idea is to prove simultaneously for all positive weights that there is a distribution on independent sets meeting the local demands; the introduction of vertex weights appears to be crucial to the method.
In the unweighted case, our results also generalize fractional chromatic number bounds obtained via local occupancy~\cite{DJKP20Local,DJKP21Occupancy,DKPS20Structure}. The resulting upper bounds on the fractional chromatic number are slightly different, but asymptotically equivalent, for example in regular graphs.
The extension to weighted graphs or graphs with given average degree, however, seems out of reach of local occupancy methods at present.
One curious feature of our results in the case of graphs with fractionally $r$-colorable neighborhoods is that they seem to make essential use the fractional colorability of the neighborhoods. Local occupancy methods derive similar, albeit weaker, results from the weaker assumption that the neighborhoods have Hall ratio at most \(r\), see~\cite{DKPS20Structure}.
We suspect that the local demands setting requires the stronger assumption, but leave this as an open question.

Recall that the maximum average degree $\mad(H)$ of a graph $H$ is
\[ \mad(H) = \max_{H'\subseteq H} \frac{2|E(H')|}{|V(H')|}, \]
where the maximum is over all subgraphs $H'$ of $H$ with at least one vertex.
We define $\mad(H)=0$ if $H$ has no vertices.
An important generalization of triangle-free graphs is the class of graphs with maximum average degree at most $a$ in neighborhoods.
We refer to this setting as the \emph{sparse neighborhoods} setting.
Taking $a=0$ recovers triangle-free graphs.
The following result is a local demand version of Li, Rousseau and Zang's independence number bounds~\cite{LRZ01,LRZ02}.
For convenience we specialize to the case of a fixed mad bound in neighborhoods, but the proof method could be adapted to handle a more general setting where the mad bound is not fixed in advance.

\begin{theorem}\label{thm:LDAsymptoticSparseNbhds}
  Fix $a\ge 0$.
  For every graph $G=(V,E)$ such that for every $v\in V$, $\mad(G[N(v)])\le a$, and for every positive weight function $w\in\RRpos^V$, there is a distribution $\mu$ on $\cI(G)$ such that for all $v\in V$,
  \[
   \mu(v)
   \ge(1-o(1))\frac{\log d_w(v)}{d_w(v)}
   \qquad\text{as }d_w(v)\to\infty.
  \]
\end{theorem}

Due to the introduction of vertex weights, our methods require a generalization of a maximum average degree bound in neighborhoods: the supremum of a weighted triangle count over all positive weight functions must be bounded.
For a graph \(H=(A,F)\) define
\begin{equation}\label{eq:intro-TD}
\begin{aligned}
 T_H(w)&=\sum_{\{x,y,z\}\in\mathcal T(H)}w(x)w(y)w(z),\\
 D_H(w)&=\sum_{xy\in F}w(x)w(y)(w(x)+w(y)),
\end{aligned}
\end{equation}
where \(\mathcal T(H)\) is the set of (unordered) triangles of \(H\).
If $w(v)=1$ for all $v\in A$, note that $T_H(w)$ is the number of triangles of $H$ and $D_H(w)=2|F|$ is the degree sum of $H$. 
Now, for a graph $G=(V,E)$ we define the triangle density parameter $\Theta(G)$ by
\begin{equation}\label{eq:intro-Theta}
 \Theta(G)=
 \max_{A\subseteq V}
 \ \sup_{w\in\RRpos^{A}}
 \begin{cases}
  6T_{G[A]}(w)/D_{G[A]}(w),&D_{G[A]}(w)>0,\\
  0,&D_{G[A]}(w)=0.
 \end{cases}
\end{equation}
In particular, \(\Theta(G)=0\) if and only if \(G\) is triangle-free.
Compared to a maximum average degree bound in neighborhoods, the supremum over vertex weights in the definition of $\Theta$ incurs some slight loss, and we show in \cref{prop:Theta} that if every neighborhood of $G$ has maximum average degree at most \(a\), then \(\Theta(G)\le 3a/2\).
Curiously, the constant $3/2$ is tight by the example of book graphs.

The precise version of \cref{thm:LDAsymptoticSparseNbhds} uses a generalization of Shearer's function first studied by Li, Rousseau and Zang~\cite{LRZ01,LRZ02}.
For \(m\ge1\), define $f_m : [0,\infty)\to (0,1]$ by $f_m(0)=1$ and for $x>0$
\begin{equation}\label{eq:fmDef}
 f_m(x)=
 \int_0^1
 \frac{(1-t)^{1/m}}{m+(x-m)t}\,\mathrm dt.
\end{equation}
Write \(f=f_1\).
Explicitly,
\begin{equation}
 f(x)=\frac{x\log x-x+1}{(x-1)^2},
\end{equation}
with continuous values \(f(0)=1\) and \(f(1)=1/2\).
This is Shearer's function~\cite{She83}, also used by Martinsson and Steiner~\cite{MS25}.

\begin{theorem}\label{thm:LDSparseNbdhs}
  Fix $b\ge 0$.
  For all graphs $G=(V,E)$ such that $\Theta(G)\le b$ and for all positive weight functions $w\in\RRpos^V$, there is a distribution $\mu$ on $\cI(G)$ such that, for all $v\in V$,
  \begin{equation}\label{eq:LDSparseNbhdsMarginal}
   \mu(v) \ge f_{b+1}(d_w(v)).
  \end{equation}
\end{theorem}

The other locally sparse setting we consider is that of graphs with fractionally \(r\)-colorable neighborhoods.
In this setting, we assume that for each vertex the subgraph induced by its neighborhood is fractionally \(r\)-colorable, for some fixed \(r\). 
This is equivalent\footnote{Perhaps a standard definition of fractional colorability gives this statement with marginals at least $1/r$, but a standard thinning argument gives the statement with marginals exactly $1/r$.} to assuming that for each $v\in V$, there exists a distribution on $\cI(G[N(v)])$ such that each vertex in the neighborhood is occupied with probability $1/r$.
We obtain the following asymptotic result, involving an increasing function \(K:[1,\infty)\to[1,\infty)\) such that \(K(1)=1\) and \(K(r)\sim\log r\) as \(r\to\infty\).
This function is defined in more detail in \cref{sec:r-colorable} (and we adopt the notation used in~\cite{DKPS20Structure}). 
Related results include the independence number bound of Alon~\cite{Alo96} whose proof is based on an alternative method of Shearer~\cite{She95a} (that was developed into local occupancy~\cite{DJPR17b,DJKP20Local,DJKP21Occupancy,DKPS20Structure}) and (list) chromatic number results in this setting~\cite{Joh96,BKNP22,DKPS20Structure}.

\begin{theorem}\label{thm:LDAsymptoticRColNbdhs}
  Fix \(r\ge1\).
  For every graph \(G=(V,E)\) whose neighborhoods are fractionally \(r\)-colorable and every positive weight function $w\in\RRpos^V$, there is a distribution \(\mu\) on \(\cI(G)\) such that, for all $v\in V$,
  \[
   \mu(v)
   \ge(1-o(1))\frac{\log d_w(v)}{K(r)d_w(v)}
   \qquad\text{as }d_w(v)\to\infty.
  \]
\end{theorem}

Dhawan's method gives a version of this result in the unweighted setting~\cite[Thm.~1.9]{Dhawan26Entropy}.
For fixed \(\eps>0\) and integer \(r\ge1\), it guarantees marginals at least
\[
 (1-\eps)\frac{\log d(v)}{8d(v)\log(2r)}
\]
on a locally \(r\)-colorable graph of sufficiently large minimum degree.
The respective leading coefficients as the degrees tend to infinity are \(1/(8\log(2r))\) and \(1/K(r)\).
Since it turns out (\cref{lem:rColorableAnalytic}) that 
\( K(r) \le 2\log(2r)\),
the improvement factor is at least \(4\), and \(8\log(2r)/K(r)\to8\) as \(r\to\infty\).
Our statement also allows arbitrary positive vertex weights.
The two results and the proof methods are nevertheless closely related, though our proof is more complex than a simple optimization of parameters in Dhawan's method.
The precise version of our result uses Shearer's function $f$ where the input is scaled by $K(r)$.

\begin{theorem}
\label{thm:LDRColNbdhs}
  Fix \(r\ge1\).
  For every graph \(G=(V,E)\) whose neighborhoods are fractionally \(r\)-colorable and every positive weight function $w\in\RRpos^V$, there is a distribution \(\mu\) on \(\cI(G)\) such that, for all $v\in V$,
\begin{equation}\label{eq:LDRColMarginal}
 \mu(v)\ge f(K(r)d_w(v)).
\end{equation}
\end{theorem}

\subsection{Lower bounds on the independence polynomial}

We also give lower bounds on the independence polynomial using similar proof techniques. 
The method actually gives bounds on the multivariate independence polynomial (defined in \cref{sec:partition}) throughout the positive orthant.
Under an average degree hypothesis, the univariate specializations we state below for the introduction improve corresponding bounds obtained via local occupancy~\cite{DJPR17b,DKPS20Structure,DJKP21Occupancy} by replacing maximum degree with average degree (albeit with the assumption of neighborhoods having Hall ratio at most \(r\) strengthened to being fractionally \(r\)-colorable in the relevant case). 
We also give strengthenings in the form of bounds in terms of the degree sequence. 
A recent development of the occupancy method~\cite{DST25,DSST26} gives bounds on the expected size of an independent set from the hard-core model in terms of the degree sequence, but it is restricted to such small $\lambda$ that the consequences are substantially weaker than those obtained here. 
In particualr, for \(a=0\) \cite[Thm.~2]{DSST26} gives exactly the case \(a=0\) of \cref{thm:IPSparseNbhdsDegrees} below for the restricted range $\lambda\le c/\Delta$ where $c$ is a small constant and $\Delta$ is the maximum degree.
We also extend an asymptotically sharper bound of Buys, van den Heuvel, and Kang~\cite{BHK25} from triangle-free graphs to graphs with maximum average degree bound $a$ in neighborhoods.
Their triangle-free result holds for $\lambda\in[0,2]$; our extension holds for $\lambda\in[0,2(a+1)]$, recovering their range when $a=0$.
Here in the introduction, we state consequences of the deferred multivariate results as bounds on the univariate independence polynomial for graphs with given degree sequence.

For brevity, we define
\[
 \mathcal A(x)=\frac{x^2}{2}+x.
\]
Let $W$ denote the principal nonnegative branch of the Lambert \(W\)-function, defined by $W(z)e^{W(z)}=z$ for $z\ge0$.

\begin{theorem}\label{thm:IPSparseNbhdsDegrees}
Fix $a\ge0$, and let $G=(V,E)$ satisfy $\mad(G[N(v)])\le a$ for every $v\in V$.
Then, for every $\lambda>0$,
\begin{equation}
 \log Z_G(\lambda)
 \ge
 \sum_{v\in V}
 \frac{\mathcal A\bigl(W((1+\lambda)^a d(v)\log(1+\lambda))\bigr)}
      {(1+\lambda)^a d(v)}.
\end{equation}
The summand for $d(v)=0$ is interpreted by continuity as $\log(1+\lambda)$.
\end{theorem}

The summand is decreasing and convex as a function of $d(v)$ (\cref{prop:summands}).
Thus, if $G=(V,E)$ has $V\ne\emptyset$ and average degree at most $d>0$, the theorem gives
\[
 \frac1{|V|}\log Z_G(\lambda)
 \ge
 \frac{\mathcal A\bigl(W(d(1+\lambda)^a
                 \log(1+\lambda))\bigr)}
      {d(1+\lambda)^a}.
\]
This has the form of the bound obtained for graphs of \emph{maximum} degree $d$ via local occupancy~\cite{DJKP21Occupancy,DKPS20Structure}; the degree sequence and average degree results here refine that bound.
The bound is asymptotically tight on triangle-free regular graphs for a reasonable range of $\lambda$, see~\cite{DJPR17b}. See also~\cite{BHK25} for an alternate discussion of sharpness and details on lack of precision in lower-order terms.

For triangle-free graphs of given average degree, Buys, van den Heuvel, and Kang~\cite{BHK25} proved an asymptotically sharper lower bound on the univariate independence polynomial for $\lambda\in[0,2]$.
For each fixed $\lambda\in(0,2]$, their bound and the preceding average-degree bound have the same leading order as the average degree tends to infinity, but the bound in \cref{thm:IPSparseNbhdsDegrees} loses precision in lower-order terms.
We remedy this by providing a generalization of their method to graphs with bounded maximum average degree in neighborhoods.

\begin{theorem}
\label{thm:IPSparseNbhdsBvdHK}
Fix \(a\ge0\) and let \(G=(V,E)\) be a graph with $V\ne\emptyset$ and average degree \(d\), such that for all $v\in V$, $\mad(G[N(v)])\le a$.
If
\begin{equation}\label{eq:IPSparseNbhdsBvdHKRange}
 0\le\lambda\le2(a+1),
\end{equation}
then
\begin{equation}\label{eq:IPSparseNbhdsBvdHKBound}
 \frac{1}{|V|}\log Z_G(\lambda)\ge
 \frac{\mathcal A(W(\lambda d))-\mathcal A(W(\lambda(a+2)))}
      {d-(a+2)},
\end{equation}
where at \(d=a+2\), the quotient is interpreted by continuity as \(W(\lambda(a+2))/(a+2)\).
\end{theorem}

The case $a=0$ is exactly the result of Buys, van den Heuvel, and Kang~\cite[Thm.~5]{BHK25} for triangle-free graphs of given average degree, and our proof method reduces precisely to theirs in this case.
For $a>0$, the permitted fugacity range extends beyond $[0,2]$.

For fractionally colorable neighborhoods we give the analog of \cref{thm:IPSparseNbhdsDegrees}.
Note that we do not pursue a version of \cref{thm:IPSparseNbhdsBvdHK} in this setting.

\begin{theorem}\label{thm:IPRColDegrees}
Fix $r\ge1$, and let $G=(V,E)$ have fractionally $r$-colorable neighborhoods.
Then, for every $\lambda>0$,
\begin{equation}\label{eq:IPRColDegreesBound}
 \log Z_G(\lambda)
 \ge
 \sum_{v\in V}
 \frac{\mathcal A\bigl(W(K(r)d(v)\log(1+\lambda))\bigr)}
      {K(r)d(v)}.
\end{equation}
The summand for $d(v)=0$ is interpreted by continuity as $\log(1+\lambda)$.
\end{theorem}

\noindent
Again, the summand is decreasing and convex in the degree (\cref{prop:summands}).
In particular, if $G=(V,E)$ has $V\ne\emptyset$ and average degree at most $d>0$, then
\[
 \frac1{|V|}\log Z_G(\lambda)
 \ge
 \frac{\mathcal A\bigl(W(K(r)d\log(1+\lambda))\bigr)}
      {K(r)d}.
\]
This gives degree-sequence and average-degree analogues of the bound derived from an occupancy fraction bound in~\cite{DKPS20Structure}.

\subsection{Applications}

Martinsson and Steiner's local Shearer bound~\cite{MS25}, which is a special case of \cref{thm:LDSparseNbdhs,thm:LDRColNbdhs}, has applications to bounding the fractional chromatic number of triangle-free graphs with a given number $n$ of vertices~\cite[Thm.~1.4]{MS25}, and in the setting of a given number $m$ of edges in addition~\cite[Thm.~1.5]{MS25}.
Our results extend these applications to graphs with bounded maximum average degree in neighborhoods and graphs with fractionally colorable neighborhoods, respectively.
Similarly, we obtain bounds on the fractional chromatic number in terms of the spectral radius in these settings, analogous to the corresponding result of Martinsson and Steiner~\cite[Thm.~5.3]{MS25}.
Given our strengthening of their theorem, the derivations are identical.
We omit the details.


\subsection{Organization}

\Cref{sec:demands} establishes \crefrange{thm:LDAsymptoticSparseNbhds}{thm:LDRColNbdhs} through a common induction argument.
\Cref{sec:partition} establishes \cref{thm:IPSparseNbhdsDegrees,thm:IPRColDegrees} through multivariate strengthenings.
A common inequality underlies both multivariate results, while a separate averaging argument in \cref{sec:partition} yields \cref{thm:IPSparseNbhdsBvdHK}.

\section{Local demands}\label{sec:demands}

\subsection{Proof overview}

Since weighted degree is unchanged when all weights are multiplied by the same positive constant, we normalize an initial weighting so that \(w(V(G))=1\) whenever convenient.
Fix a demand function \(g\) and a hereditary class of graphs.
For every graph \(G=(V,E)\) in the class and every positive weighting \(w\), we seek a distribution on \(\cI(G)\) whose marginal at \(v\) is at least \(g(d_w(v))\).
For the functions considered below, this is trivial when \(G\) has no vertices, and the proof proceeds by induction on \(|V|\).

For the induction step, let \(\delta(G)\in[0,1]\) be the least additive error such that, for every positive weighting \(w\), there is a distribution on \(\cI(G)\) whose marginal at every \(v\) is at least \(g(d_w(v))-\delta(G)\).
Standard compactness considerations give the existence of this minimum, and our goal is to prove that \(\delta(G)=0\).
By the induction hypothesis, we have \(\delta(G')=0\) for every proper induced subgraph \(G'\) of \(G\).
To bound \(\delta(G)\), consider the following construction of a random independent set in \(\cI(G)\).

An application-specific random experiment constructs an induced subgraph \(G'\subseteq G\), where we write \(G'=(V',E')\).
It also produces weights \(w'\in\RRpos^{V'}\) and an independent set \(S\in\cI(G)\) such that \(N_G[S]\cap V'=\emptyset\).
If \(G'\) is a proper induced subgraph of \(G\), then by induction we can sample \(I'\in\cI(G')\) from the distribution supplied for \(G'\), weights \(w'\), and demand function \(g\), and in this case we return \(S\cup I'\).
If \(G'=G\), then necessarily \(S=\emptyset\), and we instead use the distribution for the positive weights \(w'\), with additive error \(\delta(G)\) in the marginals, supplied by its definition.
Let the resulting distribution on \(\cI(G)\) be denoted by \(\mu\).
A careful analysis of its marginals bounds \(\delta(G)\); the technical heart of each application is to choose the experiment so that this bound forces \(\delta(G)=0\).

\subsection{General proof structure}

When the graph and weights must be explicit, we write
\[
 d_{G,w}(v)=\frac{w(N_G(v))}{w(v)}
\]
for the weighted degree of a vertex \(v\) in $G$.
For \(A\subseteq V(G)\), write
\[
 N_G(A)=\bigcup_{x\in A}N_G(x),
 \qquad
 N_G[A]=A\cup N_G(A).
\]

Fix a decreasing function \(g:[0,\infty)\to(0,1]\).
Though it is not formally required in the abstract setup, in applications we will also require that $g$ is convex.
For a graph $G=(V,E)$, define \(\delta(G)\) to be the least error \(\delta\in[0,1]\) such that, for every \(w\in\RRpos^{V}\), there is a distribution \(\mu\) on \(\cI(G)\) satisfying for all $v\in V$
\begin{equation}\label{eq:minimumError}
 \mu(v)\ge g(d_{G,w}(v))-\delta.
\end{equation}
This least value exists.
Indeed, the set of admissible errors is nonempty because it contains one.
Let \((\delta_j)_{j\in\NN}\) be a sequence of admissible errors decreasing to their infimum \(\delta\).
For each fixed \(w\), compactness of the probability simplex on the finite set \(\cI(G)\) gives a subsequential limit of the corresponding probability distribution, and that limit satisfies~\eqref{eq:minimumError} with error \(\delta\).
Since \(w\) was arbitrary, \(\delta\) is itself admissible.

We use the following argument during an induction.
A random experiment on \((G,w)\) produces an independent set \(S\in\cI(G)\), an induced subgraph \(G'\subseteq G\) such that
\[
 N_G[S]\cap V(G')=\emptyset,
\]
and positive weights \(w'\in\RRpos^{V(G')}\).
The equality \(G'=G\) is allowed and forces \(S=\emptyset\).
For a target vertex \(v\in V(G)\), define its random ideal selection probability by
\begin{equation}\label{eq:idealSelection}
 Y_v=
 \begin{cases}
  1,&v\in S,\\
  g(d_{G',w'}(v)),&v\in V(G'),\\
  0,&\text{otherwise}.
 \end{cases}
\end{equation}
The next lemma gives the structure of the inductive step.

\begin{lemma}\label{lem:LDInductionDriver}
Suppose that, for every proper induced subgraph \(G'\subsetneq G\) and all positive weights \(w'\in \RRpos^{V(G')}\), there is a distribution $\mu'$ on \(\cI(G')\) such that for $v\in V(G')$ we have \(\mu'(v) \ge g(d_{G',w'}(v))\).
Suppose further that there are constants \(h\ge0\) and \(\rho\in[0,1)\) such that, for every positive weighting \(w\in \RRpos^{V(G)}\), the random experiment on \((G,w)\) can be chosen so that, simultaneously for every \(v\in V(G)\),
\begin{equation}
 \EE Y_v\ge g(d_{G,w}(v))-h
\end{equation}
and \(\PP(G'=G)\le \rho\).
Then
\begin{equation}\label{eq:deltaContraction}
 \delta(G)\le\frac{h}{1-\rho}.
\end{equation}
\end{lemma}

\begin{proof}
The assertion is immediate when \(G\) is empty, so suppose that \(G\) is nonempty.
Fix arbitrary positive weights \(w\) on \(G\), and normalize them by setting \(\bar w=w/w(V(G))\).
Apply the random experiment to \((G,\bar w)\), producing \((S,G',w')\) and the ideal selection probabilities \(Y_v\) for \(v\in V(G)\).
After each outcome with \(G'\subsetneq G\), sample an independent set from the assumed distribution on \(\cI(G')\) and return its union with \(S\).
If \(G'=G\), use the distribution for the positive weights \(w'\) with additive error \(\delta(G)\) on the marginals.
This gives a distribution \(\mu\) on \(\cI(G)\).
By construction, for every \(v\in V(G)\),
\[
 \mu(v)\ge\EE Y_v-\delta(G)\PP(G'=G)
 \ge g(d_{G,\bar w}(v))-h-\rho\delta(G)
 =g(d_{G,w}(v))-h-\rho\delta(G).
\]
Since the construction works for every positive \(w\), the definition of \(\delta(G)\) gives
\[
 \delta(G)\le h+\rho\delta(G).
\]
Rearranging gives~\eqref{eq:deltaContraction}.
As a minor technical subtlety, the desired inequality follows from $\delta(G)\le 1$ if \(h+\rho\delta(G)\ge1\), otherwise \(h+\rho\delta(G)\) is an admissible error and it follows from the random construction detailed above.
\end{proof}

\subsection{Sparse neighborhoods}\label{sub:sparseNbhds}

In order to prove \cref{thm:LDAsymptoticSparseNbhds} we apply \cref{thm:LDSparseNbdhs}, but this requires an understanding of how bounded mad in neighborhoods implies bounded triangle density.

\begin{proposition}\label{prop:Theta}
The parameter \(\Theta(G)\) is finite, and it is zero exactly when \(G\) is triangle-free.
Moreover, if \(\mad(G[N(v)])\le a\) for every \(v\in V(G)\), then
\begin{equation}\label{eq:Theta-mad}
 \Theta(G)\le\frac{3a}{2}.
\end{equation}
The factor \(3/2\) is sharp.
\end{proposition}

\begin{proof}
  Recall the definitions~\eqref{eq:intro-TD} the weighted triangle count $T$ and edge count $D$, as well as the definition of $\Theta(G)$ in~\eqref{eq:intro-Theta}.
For a triangle \(xyz\), the AM--GM inequality gives
\[
 6w(x)w(y)w(z)
 \le
 w(x)w(y)(w(x)+w(y))+w(x)w(z)(w(x)+w(z))
 +w(y)w(z)(w(y)+w(z)).
\]
An edge lies in at most \(|V(G)|-2\) triangles, so summation proves finiteness.
If \(G\) is triangle-free, every numerator defining \(\Theta(G)\) is zero.
Conversely, if \(xyz\) is a triangle, take \(A=\{x,y,z\}\) and any positive weights on \(A\); then the numerator is positive, so \(\Theta(G)>0\).

For any graph \(F\) with \(\mad(F)\le a\) and every \(r\in\RRnn^{V(F)}\),
\begin{equation}\label{eq:vector-mad}
 2\sum_{xy\in E(F)}\min\{r_x,r_y\}
 \le a\sum_{x\in V(F)}r_x.
\end{equation}
Indeed, integrate \(2|E(F[S_t])|\le a|S_t|\) over \(S_t=\{x:r_x\ge t\}\).
Fix \(A\subseteq V(G)\) and positive weights \(w\) on \(A\).
Apply~\eqref{eq:vector-mad} to \(G[N(v)\cap A]\) with \(r_x=w(x)\), multiply by \(w(v)^2\), and sum over \(v\in A\).
The right side is \(aD_{G[A]}(w)\).
A triangle with weights \(r\ge s\ge t\) contributes at least
\[
 2(r^2t+s^2t+t^2s)
 =2rst\left(\frac rs+\frac sr+\frac tr\right)
 \ge4rst
\]
to the left side.
Thus \(4T_{G[A]}(w)\le aD_{G[A]}(w)\), which proves \eqref{eq:Theta-mad}.

For sharpness, let \(B_k\) be the book graph formed by \(k\) triangles sharing one edge.
Its largest neighborhood maximum average degree is \(2k/(k+1)\to2\).
Give each endpoint of the common edge weight one and each other vertex weight \(k^{-1/2}\).
Then
\[
 T_{B_k}(w)=\sqrt{k},
 \qquad
 D_{B_k}(w)=2\sqrt{k}+4,
\]
so
\[
 \Theta(B_k)\ge\frac{6\sqrt{k}}{2\sqrt{k}+4}\longrightarrow3.
\]
On the other hand,~\eqref{eq:Theta-mad} and the preceding value of the largest neighborhood maximum average degree give \(\Theta(B_k)\le3k/(k+1)\to3\).
Hence \(\Theta(B_k)\to3\), and the ratio in~\eqref{eq:Theta-mad} tends to \(3/2\).
\end{proof}

\subsubsection{The sparse neighborhoods induction}

We first define the random experiment we use in this setting.
Fix \(\beta>0\), a nonempty graph \(G=(V,E)\) with \(\Theta(G)<\beta\), and positive weights \(w\in \RRpos^V\), normalized so that \(w(V)=1\).
The normalization is without loss of generality and simplifies the notation.
Define
\[
 n=|V|,\qquad
 d_z=d_{G,w}(z),\qquad
 g=f_{\beta+1}.
\]
The experiment uses a vector \(u\in[0,2n]^{V}\) which we construct below from \(G,w,\beta\).
For a fixed \(0<\eps<1\), define weights \(w'\) on \(G\) by
\begin{equation}\label{eq:sparse-tilt}
 w'(z)=w(z)e^{\eps u_z}
 \qquad(z\in V).
\end{equation}
No matter the choice of $G'$, we take these weights and implicitly make the restriction to the vertex set of $G'$.
The experiment then produces \((S,G',w')\) as follows:
\begin{itemize}
\item with probability \(1-\eps\), take
\[
 S=\emptyset,\qquad G'=G;
\]
\item with probability \(\eps\), choose a pivot vertex \(x\) with probability \(w(x)\), and take
\[
 S=\{x\},\qquad
 G'=G-N_G[x].
\]
\end{itemize}
In either case \(N_G[S]\cap V(G')=\emptyset\), and we have \(\PP(G'=G)=1-\eps\).
For a target vertex \(v\), the ideal selection probability from~\eqref{eq:idealSelection} therefore satisfies
\begin{align}
\EE Y_v
 ={}&(1-\eps)g(d_{G,w'}(v))
 +\eps w(v)+\eps\sum_{x\notin N_G[v]}w(x)
 g(d_{G-N_G[x],w'}(v)).
 \label{eq:sparse-experiment-value}
\end{align}
The analytic task is to choose \(u\) so that, uniformly over \(w\) and \(v\),
\[
 \EE Y_v\ge g(d_v)-r_{n,\beta}(\eps)
 \qquad\text{with}\qquad
 r_{n,\beta}(\eps)=o(\eps).
\]
\Cref{lem:LDInductionDriver} will then give \(\delta(G)\le r_{n,\beta}(\eps)/\eps\), and we can take a sequence with $\eps \to 0$.

The proof relies on two ingredients.
\Cref{lem:fm} supplies the needed analytic properties of \(g=f_{\beta+1}\), while \cref{lem:sparse-obstacle-data} below controls terms in the computation arising from triangles.
Their proofs are deferred until after the main calculation.

\begin{lemma}\label{lem:fm}
For every \(m\ge1\), \(f_m\) is continuous on \([0,\infty)\), positive, decreasing, and convex.
It is \(C^2\) on \((0,\infty)\) and satisfies
\begin{equation}\label{eq:fm-ode}
 x(x-m)f_m'(x)+(x+1)f_m(x)=1
 \qquad(x>0).
\end{equation}
For fixed \(x\ge0\), \(f_m(x)\) is continuous and nonincreasing in \(m\).
For fixed \(m\), as $x\to \infty$ we have 
\begin{equation}\label{eq:fm-asymptotic}
 f_m(x)\sim\frac{\log(x/m)}x\sim\frac{\log x}{x}.
\end{equation}
If \(g=f_m\) and \(F(x)=g(e^x)\), then as $x\to-\infty$ or \(x\to\infty\),
\begin{equation}\label{eq:fm-log-derivative}
 F'(x)=e^x g'(e^x)\longrightarrow0.
\end{equation}
In particular, \(F'\) is bounded and uniformly continuous on \(\RR\).
More precisely, if
\[
 \omega_m(\rho)
 :=\sup_{\substack{x\in\RR\\ |h|\le\rho}}
   |F'(x+h)-F'(x)|,
\]
then \(\omega_m(\rho)\to0\) as \(\rho\downarrow0\), and, for every \(d>0\) and \(h\in\RR\),
\begin{equation}\label{eq:fm-log-taylor}
 \left|g(de^h)-g(d)-h d g'(d)\right|
 \le |h|\omega_m(|h|).
\end{equation}
\end{lemma}

We continue following Martinsson and Steiner's approach, extending it to graphs with triangles where necessary.
We use the convexity of $g$ to combine the terms in~\eqref{eq:sparse-experiment-value}, and then Taylor expand around $\eps=0$.

First, note that if the target vertex $v$ is isolated, then $\EE Y_v=1 = g(0)$ by direct computation, and the desired inequality holds.
We therefore assume that $v$ is nonisolated.

Since
\[ 1-\eps + \eps \sum_{x\notin N_G[v]} w(x) = 1 - \eps w(N_G[v]), \] by the convexity of \(g\) we have
\begin{equation}\label{eq:sparse-jensen-bound}
 \EE Y_v
 \ge \eps w(v)+\left(1-\eps w(N_G[v])\right)g(t_\eps),
\end{equation}
where
\[
t_\eps = \frac{(1-\eps)w'(N_G(v)) + \eps\sum_{x\notin N_G[v]}w(x)w'(N_{G-N_G[x]}(v))}{w'(v)(1-\eps w(N_G[v]))}.
\]
Expanding the numerator of $t_\eps$, we have
\begin{align*}
  &(1-\eps)w'(N_G(v)) + \eps\sum_{x\notin N_G[v]}w(x)w'(N_{G-N_G[x]}(v)) \\
  &= \sum_{y\in N_G(v)}(1-\eps)w'(y) + \sum_{x\notin N_G[v]} \eps w(x) \sum_{y\in N_G(v)\setminus N_G[x]} w'(y) \\
  &= \sum_{y\in N_G(v)}\left((1-\eps) + \sum_{x\in V\setminus (N_G[v]\cup N_G[y])}\eps w(x)\right)w'(y) \\
  &= \sum_{y\in N_G(v)}\left(1-\eps w(N_G[v]\cup N_G[y])\right)w'(y).
\end{align*}
For $y\in N_G(v)$ we have $N_G[v]\cup N_G[y]=N_G(v)\cup N_G(y)$, and by inclusion-exclusion we then have
\begin{align*}
  t_\eps &= \frac{1}{w'(v)} \sum_{y\in N_G(v)} \frac{1-\eps w(N_G(v)) - \eps w(N_G(y)) + \eps w(N_G(v)\cap N_G(y))}{1-\eps w(N_G[v])}w'(y)
\end{align*}
The term $\eps w(N_G(v)\cap N_G(y))$ is the new term that arises in the presence of triangles; it is zero if $v$ is contained in no triangles.
Since $w'(y)=w(y)e^{\eps u_y}$, Taylor expansion at \(\eps=0\) gives
\begin{equation}\label{eq:sparse-t-first-variation}
\begin{aligned}
 t_\eps
 ={}&d_v+\frac{\eps}{w(v)}\bigg(
  \sum_{y\in N_G(v)}w(y)\bigl(u_y-w(N_G(y))\bigr)
  -w(N_G(v))u_v \\[-2mm]
 &\hspace{24mm}{}+w(v)w(N_G(v))
  +\sum_{y\in N_G(v)}w(y)w(N_G(v)\cap N_G(y))
 \bigg)
 +O_n(d_v\eps^2).
\end{aligned}
\end{equation}
For brevity, we write $t_\eps = d_v + \eps t'_0 + O_n(d_v\eps^2)$, where $t'_0$ is the coefficient of $\eps$ in the expansion above.
To see that the remainder is uniform, write
\[
 r_y=\frac{w(y)}{w(N_G(v))},\qquad
 s_v=w(N_G[v]),\qquad
 a_{v,y}=w(N_G(v)\cup N_G(y)).
\]
The exact expression above can then be written as
\[
 \frac{t_\eps}{d_v}
 =\sum_{y\in N_G(v)}r_y e^{\eps(u_y-u_v)}
   \frac{1-\eps a_{v,y}}{1-\eps s_v}.
\]
Here \(0\le s_v\le a_{v,y}\le1\) and \(|u_y-u_v|\le2n\).
For \(0\le\eps\le1/2\), each summand has second derivative bounded by a constant depending only on \(n\).
Taylor's theorem and \(\sum_yr_y=1\) therefore give the stated \(O_n(d_v\eps^2)\) remainder uniformly over (normalized) weights \(w\), all nonisolated targets $v$, and \(u\in[0,2n]^{V}\).

We now give names to the two terms that determine the choice of $u$.
The final sum in~\eqref{eq:sparse-t-first-variation} is
\[
 \sum_{y\in N_G(v)}w(y)w(N_G(v)\cap N_G(y))
 =2\sum_{yz\in E(G[N_G(v)])}w(y)w(z).
\]
For \(z\in V\), to handle the term involving triangles we set
\begin{align}
 T_z&=2\sum_{xy\in E(G[N_G(z)])}w(x)w(y),
 \label{eq:rooted-T}
\end{align}
so that triangles contribute \(T_v/w(v)\) to \(t'_0\).

The form of~\eqref{eq:sparse-t-first-variation} suggests taking \(u_z=w(N_G(z))+\eta_z\) for some triangle-specific correction $\eta\in\RR^V$ to Martinsson and Steiner's weight drift.
Define the weighted Laplacian \(L\) of \(G\) by
\[
 (L\eta)_z
 =\sum_{y\in N_G(z)}w(z)w(y)(\eta_z-\eta_y).
\]
The identity
\[
 \sum_{y\in N_G(v)}w(y)\eta_y-w(N_G(v))\eta_v
 =-\frac{(L\eta)_v}{w(v)}
\]
reduces the coefficient in~\eqref{eq:sparse-t-first-variation} to
\begin{equation}\label{eq:sparse-t-compensated}
 t'_0
 =\frac{T_v}{w(v)}+w(v)d_v(1-d_v)
  -\frac{(L\eta)_v}{w(v)^2}.
\end{equation}
To help match terms with the differential equation for \(g=f_{\beta+1}\), we write this in terms of
\begin{equation}\label{eq:sparse-p}
 p_z:=w(z)T_z-\beta w(z)^3d_z.
\end{equation}
Then~\eqref{eq:sparse-t-compensated} becomes
\begin{equation}\label{eq:sparse-t-defect}
 t'_0
 =\frac{p_v-(L\eta)_v}{w(v)^2}
  +w(v)d_v(\beta+1-d_v).
\end{equation}
The first term is handled by \cref{lem:sparse-obstacle-data} below, and the second is handled by the differential equation for \(g\).
To compensate for the first term, it turns out to be enough to construct \(\eta\ge0\) such that \(L\eta\ge p\).
This would imply
\begin{equation}\label{eq:sparse-t-variation-bound}
 t'_0
 \le w(v)d_v(\beta+1-d_v),
\end{equation}
which lets us continue the computation via the differential equation.
The following lemma guarantees the existence of such a vector \(\eta\) under the condition $\Theta(G)<\beta$.
\begin{lemma}
\label{lem:sparse-obstacle-data}
There is a vector \(\eta\in\RRnn^{V(G)}\), zero at every isolated vertex, such that
\begin{equation}\label{eq:sparse-compensator}
 L\eta\ge p.
\end{equation}
Moreover, the choice
\begin{equation}\label{eq:sparse-direction}
 u_z=w(N_G(z))+\eta_z
 \qquad(z\in V(G))
\end{equation}
satisfies
\begin{equation}\label{eq:sparse-direction-bound}
 0\le u_z\le2n
 \qquad(z\in V(G)).
\end{equation}
\end{lemma}

\noindent
Assuming \cref{lem:fm,lem:sparse-obstacle-data}, the following lemma completes the proof of the main inequality for the sparse setting.

\begin{lemma}
\label{lem:sparse-tilted-deletion}
For the choice of \(u\) in~\eqref{eq:sparse-direction}, there is a function \(r_{n,\beta}:(0,1)\to[0,\infty)\), depending only on \(n\) and \(\beta\), such that \(r_{n,\beta}(\eps)=o(\eps)\) and, for every target vertex \(v\),
\begin{equation}\label{eq:sparse-tilted-deletion}
 \EE Y_v\ge g(d_v)-r_{n,\beta}(\eps).
\end{equation}
\end{lemma}

\begin{proof}
Choose \(u\) as in~\eqref{eq:sparse-direction}.
The desired bound is exact when \(v\) is isolated, so suppose that \(v\) is nonisolated.
After dividing the coefficient in \eqref{eq:sparse-t-first-variation} by \(d_v\), its four terms have absolute values at most \(2n\), \(2n\), \(1\), and \(1\), respectively.
Thus
\[
 |t'_0|\le(4n+2)d_v.
\]
Together with the uniform remainder in \eqref{eq:sparse-t-first-variation}, this yields
\begin{equation}\label{eq:sparse-log-expansion}
 \log\frac{t_\eps}{d_v}
 =\eps\frac{t'_0}{d_v}+O_n(\eps^2)
 =O_n(\eps).
\end{equation}
Here and below, all asymptotic error terms are uniform over the graph, the normalized weight function, and the target vertex.
Apply \eqref{eq:fm-log-taylor} with \(d=d_v\) and \(h=\log(t_\eps/d_v)\).
Equation~\eqref{eq:sparse-log-expansion} and the boundedness of \(xg'(x)\) give
\begin{equation}\label{eq:sparse-composed-expansion}
 g(t_\eps)
 =g(d_v)+\eps g'(d_v)t'_0+o_{n,\beta}(\eps),
 \qquad
 |g'(d_v)t'_0|=O_{n,\beta}(1).
\end{equation}

Substituting~\eqref{eq:sparse-composed-expansion} into \eqref{eq:sparse-jensen-bound}, and using \(w(N_G[v])=w(v)(1+d_v)\), gives
\begin{align*}
 \EE Y_v
 &\ge g(d_v)
 +\eps\left[
   w(v)-w(N_G[v])g(d_v)+g'(d_v)t'_0
 \right]
 +o_{n,\beta}(\eps)\\
 &\ge g(d_v)
 +\eps w(v)\left[
   1-(1+d_v)g(d_v)
   +d_v(\beta+1-d_v)g'(d_v)
 \right]
 +o_{n,\beta}(\eps)\\
 &=g(d_v)+o_{n,\beta}(\eps).
\end{align*}
The second inequality uses \(g'(d_v)\le0\) and \eqref{eq:sparse-t-variation-bound}; the equality is \eqref{eq:fm-ode} with \(m=\beta+1\).

The last error term is uniform, so its absolute value is bounded by some \(r_{n,\beta}(\eps)=o(\eps)\) for all sufficiently small \(\eps\).
Define \(r_{n,\beta}(\eps)=1\) for the remaining values of \(\eps\); then~\eqref{eq:sparse-tilted-deletion} follows there from \(0\le\EE Y_v\) and \(g\le1\).
This proves the lemma.
\end{proof}

Completing the proof of \cref{thm:LDSparseNbdhs} from here is a simple matter of stating the induction and taking limits.

\begin{proof}[Proof of \cref{thm:LDSparseNbdhs}]
Fix first \(\beta>b\).
We prove by induction on \(|V(G)|\) that every graph \(G\) with \(\Theta(G)<\beta\) and every positive weighting \(w\) admit a law with marginals at least \(g(d_{G,w}(v))\), where \(g=f_{\beta+1}\).
The base case of the empty graph is immediate.
The parameter \(\Theta\) is hereditary, so the induction hypothesis applies to every proper induced subgraph \(G'\) of \(G\).
By scale invariance, it suffices at the induction step to consider normalized weights which sum to $1$.

Let \(n=|V(G)|>0\).
For each \(0<\eps<1\) and each normalized \(w\), \cref{lem:sparse-tilted-deletion} gives \(\EE Y_v\ge g(d_{G,w}(v))-r_{n,\beta}(\eps)\), while \(\PP(G'=G)=1-\eps\).
Apply \cref{lem:LDInductionDriver} with \(h=r_{n,\beta}(\eps)\) and \(\rho=1-\eps\).
It gives
\[
 \delta(G)\le\frac{r_{n,\beta}(\eps)}{\eps}.
\]
Since \(r_{n,\beta}(\eps)=o(\eps)\), letting \(\eps\downarrow0\) gives \(\delta(G)=0\), completing the induction.

Apply this auxiliary result to the graph \(G\) and weighting \(w\) from the theorem.
Choose a sequence \(\beta_j\downarrow b\) with \(\beta_j>b\), and let \(\mu_j\) be the corresponding distributions on $\cI(G)$.
Along a subsequence, the laws converge in the probability simplex on \(\cI(G)\).
The parameter continuity in \cref{lem:fm} gives
\[
 f_{\beta_j+1}(d_{G,w}(v))
 \longrightarrow f_{b+1}(d_{G,w}(v)).
\]
The limiting distribution gives~\eqref{eq:LDSparseNbhdsMarginal}.
\end{proof}

\noindent
The asymptotic statement in \cref{thm:LDAsymptoticSparseNbhds} follows from the preceding theorem, \cref{prop:Theta}, and the asymptotics of \(f_{b+1}\) in \cref{lem:fm}.

\begin{proof}[Proof of \cref{thm:LDAsymptoticSparseNbhds}]
By \cref{prop:Theta}, we can apply \cref{thm:LDSparseNbdhs} with \(b=3a/2\).
Equation~\eqref{eq:fm-asymptotic} gives
\[
 f_{b+1}(d)= (1-o(1))\frac{\log d}{d}.
\]
The result now follows from \cref{thm:LDSparseNbdhs}.
\end{proof}

\subsubsection{Proofs of \cref{lem:fm,lem:sparse-obstacle-data}}

\begin{proof}[Proof of \cref{lem:fm}]
The function $f_m$ is exactly the one studied by Li--Rousseau--Zang.
They prove that it is completely monotone on \((0,\infty)\) and satisfies~\eqref{eq:fm-ode} \cite[Lem.~1]{LRZ02}.
In particular, it is positive, decreasing, convex, and \(C^2\) there.
They also prove that \(f_m(x)\) decreases with \(m\) for each fixed \(x>0\)~\cite[Lem.~2]{LRZ02}, and record \eqref{eq:fm-asymptotic}.

It remains to check the endpoint and continuity assertions and the estimates in logarithmic coordinates.
Fix \(m\ge1\).
At \(x=0\), the integrand in~\eqref{eq:fmDef} is
\[
 \frac1m(1-t)^{1/m-1},
\]
whose improper integral is one.
Since \(m+(x-m)t=m(1-t)+xt\ge m(1-t)\), this integrable function dominates the integrand for every \(x\ge0\).
Dominated convergence proves continuity at zero and extends the stated monotonicity and convexity to \([0,\infty)\).

For fixed \(x>0\), the integrand is pointwise continuous in \(m\).
Locally in \(m\), its denominator \(m(1-t)+xt\) is bounded away from zero, so dominated convergence proves continuity in \(m\).
At \(x=0\), this is immediate from \(f_m(0)=1\).

Finally, let \(g=f_m\), \(F(s)=g(e^s)\), and \(k=-g'\).
Since \(g\) is decreasing and convex, \(k\) is nonnegative and nonincreasing, so for every \(x>0\),
\[
 xk(x)
 \le2\int_{x/2}^{x}k(y)\,\mathrm dy
 =2\bigl(g(x/2)-g(x)\bigr).
\]
The right-hand side tends to zero as \(x\downarrow0\), by continuity at zero, and as \(x\to\infty\), by~\eqref{eq:fm-asymptotic}.
Consequently
\[
 F'(s)=e^s g'(e^s)\longrightarrow0
 \qquad(s\to-\infty\ \text{or}\ s\to\infty).
\]
Because \(F'\) is continuous, these limits imply that it is bounded and uniformly continuous on \(\RR\), and hence \(\omega_m(\rho)\to0\) as \(\rho\downarrow0\).

For \(d>0\), \(h\in\RR\), and \(s_0=\log d\), the fundamental theorem of calculus now gives
\begin{align*}
 \left|g(de^h)-g(d)-h d g'(d)\right|
 &=\left|\int_0^h
   \bigl(F'(s_0+s)-F'(s_0)\bigr)\,\mathrm ds\right|\\
 &\le |h|\omega_m(|h|),
\end{align*}
which is~\eqref{eq:fm-log-taylor}.
\end{proof}

\begin{proof}[Proof of \cref{lem:sparse-obstacle-data}]
We work separately on each component.
If \(z\) is isolated, then \(d_z=T_z=p_z=0\), so set \(\eta_z=0\).
Now let \(C\) be a component containing an edge, and set \(\ell=|C|\) and \(W_C=w(C)\).
For \(S\subseteq C\), write
\[
 c_C(S)=\sum_{\substack{xy\in E(G[C])\\x\in S,\ y\notin S}}w(x)w(y).
\]
Each triangle contributes twice its weight product at each of its three vertices, whereas each edge contributes its two endpoint terms.
It follows that
\begin{equation}\label{eq:sparse-component-total}
 \sum_{z\in C}p_z
 =6T_{G[C]}(w)-\beta D_{G[C]}(w)
 \le\bigl(\Theta(G)-\beta\bigr)D_{G[C]}(w)<0.
\end{equation}
The last inequality is strict because \(C\) contains an edge and all weights are positive.
We also have, for every \(S\subseteq C\),
\begin{equation}\label{eq:sparse-component-cut}
 \sum_{z\in S}p_z\le2W_Cc_C(S).
\end{equation}
Indeed, the contribution from the edges and triangles wholly in \(S\) is
\[
 6T_{G[S]}(w)-\beta D_{G[S]}(w)\le0
\]
by the definition of \(\Theta(G)\), and we may discard the remaining negative edge terms.
A triangle meeting both sides of the cut contributes twice its weight product if it has one vertex in \(S\), and four times its weight product if it has two.
Charge these amounts equally to the triangle's two cut edges.
A fixed cut edge \(xy\) then receives at most
\[
 2w(x)w(y)\sum_{z\in N_G(x)\cap N_G(y)}w(z)
 \le2W_Cw(x)w(y).
\]
Summing over the cut proves~\eqref{eq:sparse-component-cut}.

On \(\RRnn^C\), minimize
\[
 \Phi(\eta)=\frac12\sum_{xy\in E(G[C])}
 w(x)w(y)(\eta_x-\eta_y)^2-\sum_{x\in C}p_x\eta_x.
\]
A minimizer exists.
To see this, write \(\eta=t\boldsymbol1+\eta^\circ\), where \(t=\min_x\eta_x\) and \(\min_x\eta_x^\circ=0\).
Since \(C\) is connected, the quadratic part of \(\Phi\) has a positive minimum on
\[
 \{\eta^\circ\ge0:\min_x\eta_x^\circ=0,\ \|\eta^\circ\|_2=1\}.
\]
By homogeneity and Cauchy--Schwarz, there is therefore some \(c>0\) such that
\[
 \Phi(\eta)\ge c\|\eta^\circ\|_2^2-\|p\|_2\|\eta^\circ\|_2
 -t\sum_{x\in C}p_x.
\]
By~\eqref{eq:sparse-component-total}, the right-hand side tends to infinity whenever \(\eta\) is unbounded in \(\RRnn^C\), so \(\Phi\) is coercive there.

Let \(\eta\) be a minimizer.
Coordinatewise optimality gives
\begin{equation}\label{eq:sparse-obstacle-optimality}
 L\eta\ge p,
 \qquad
 \eta_x\bigl((L\eta)_x-p_x\bigr)=0
 \quad(x\in C).
\end{equation}
Moreover, \(\min_{x\in C}\eta_x=0\).
Otherwise, subtracting \(t\boldsymbol1\), where \(t=\min_x\eta_x>0\), preserves feasibility and changes the objective by \(t\sum_{x\in C}p_x<0\), contradicting minimality.

Order \(C=\{x_1,\ldots,x_\ell\}\) so that \(\eta_{x_1}\le\cdots\le\eta_{x_\ell}\).
If \(\Delta_i=\eta_{x_{i+1}}-\eta_{x_i}>0\), set \(S_i=\{x_{i+1},\ldots,x_\ell\}\).
Every coordinate in \(S_i\) is positive, so~\eqref{eq:sparse-obstacle-optimality} and cancellation of the internal edge terms give
\[
 \sum_{x\in S_i}p_x
 =\sum_{\substack{xy\in E(G[C])\\x\in S_i,\ y\notin S_i}}
   w(x)w(y)(\eta_x-\eta_y)
 \ge\Delta_i c_C(S_i).
\]
Combining this with~\eqref{eq:sparse-component-cut} gives \(\Delta_i c_C(S_i)\le2W_Cc_C(S_i)\).
The cut is nonempty because \(C\) is connected, and all weights are positive, so \(\Delta_i\le2W_C\).
Since the minimum coordinate is zero, summing the at most \(\ell-1\) positive gaps yields
\[
 0\le\eta_x\le2W_C(\ell-1)
 \qquad(x\in C).
\]

Combining these componentwise minimizers gives a nonnegative vector \(\eta\), zero at isolated vertices, satisfying \eqref{eq:sparse-compensator}.
Finally, if \(z\in C\), then
\[
 0\le u_z=w(N_G(z))+\eta_z
 \le W_C+2W_C(\ell-1)=(2\ell-1)W_C\le2n,
\]
because \(W_C\le w(V)=1\) and \(\ell\le n\).
For isolated \(z\), we have \(u_z=0\).
This proves~\eqref{eq:sparse-direction-bound}.
\end{proof}

\subsection{Fractionally colorable neighborhoods}
\label{sec:r-colorable}

We start the proofs of \cref{thm:LDAsymptoticRColNbdhs,thm:LDRColNbdhs} by defining the required random experiment in terms of an auxiliary parameter $\theta$ over which we will optimize later.
Fix \(r\ge1\) and \(0<\theta<1\), and set
\begin{equation}\label{eq:r-local-parameters}
 A=\frac r{1-\theta},
 \qquad
 g(d)=f\left(\frac{\log A}{\theta}d\right).
\end{equation}
Let \(G=(V,E)\) be a graph whose neighborhoods are fractionally \(r\)-colorable, and let \(w\in\RRpos^V\) be normalized positive weights on \(G\).
For each \(x\in V(G)\), let \(\nu_x\) be a distribution on \(\cI(G[N_G(x)])\) in which every vertex has marginal \(1/r\); the existence of such a distribution follows easily from the definition of fractional colorability.

Choose a pivot \(x\) with probability \(w(x)\).
Conditional on \(x\), produce \((S,G',w')\) as follows:
\begin{itemize}
\item with probability \(\theta\), select the pivot and take
\[
 S=\{x\},\qquad
 G'=G-N_G[x],\qquad
 w'=w|_{V(G')};
\]
\item with probability \(1-\theta\), draw \(J_x\sim\nu_x\), select no vertex, and take
\[
 S=\emptyset,\qquad
 G'=G-\bigl(N_G(x)\setminus J_x\bigr),\qquad
 w'(u)=
 \begin{cases}
  Aw(u),&u\in J_x,\\
  w(u),&u\in V(G')\setminus J_x.
 \end{cases}
\]
\end{itemize}
In both branches \(N_G[S]\cap V(G')=\emptyset\).
The selection branch always produces a proper induced subgraph, so
\begin{equation}\label{eq:r-full-graph-probability}
 \PP(G'=G)\le1-\theta.
\end{equation}
The other branch, which zooms in on an independent set in the neighborhood of the pivot, is called the \emph{certificate branch}.
Once a demand function \(g\) is fixed, let \(Y_v\) be the ideal value from~\eqref{eq:idealSelection}.
The analysis below chooses \(g\) so that the experiment satisfies, for all $v\in V$, $\EE Y_v\ge g(d_{G,w}(v))$.

For \(r\ge1\), let \(K(r)\) be the unique \(K\ge1\) satisfying
\begin{equation}\label{eq:K-equation}
 K-1=\log(rK).
\end{equation}
The proof uses the following analytic lemma, which both identifies the optimal constant and gives the inequality needed by the induction. It is curious that the variational characterization of this approach defining the constant $K(r)$ optimal for the method gives exactly the same $K(r)$ as different methods in which the variational characterization is not as prominent~\cite{DKPS20Structure}.

\begin{lemma}
\label{lem:rColorableAnalytic}
For $0<\theta<1$, recall the definitions~\eqref{eq:r-local-parameters} of $A$ and $g$.
The function \(g\) is decreasing and convex. 
For every \(d\ge0\), we have
\begin{equation}\label{eq:scaled-shearer}
 \theta+(1-\theta)g\left(\frac d{1-\theta}\right)
 -(1+d)g(d)+\frac dA g(d/A)\ge0.
\end{equation}
Moreover,
\begin{equation}\label{eq:r-optimum}
 \inf_{0<\theta<1}
 \frac{\log(r/(1-\theta))}{\theta}=K(r) \le 2\log(2r).
\end{equation}
If \(r>1\), the minimum is attained at
\begin{equation}\label{eq:r-optimizer}
 \theta_*=1-\frac1{K(r)}.
\end{equation}
For \(r=1\), the infimum is approached as \(\theta\downarrow0\).
Finally, \(K\) is increasing and
\begin{equation}\label{eq:K-asymptotic}
 K(r)=\log r+\log\log r+1+o(1)
 \qquad(r\to\infty).
\end{equation}
\end{lemma}

\begin{proof}
The monotonicity and convexity of \(g\) follow from \cref{lem:fm} and the positive rescaling of its argument.
Let \(\psi(z)=zf(z)\) and \(Q(s)=\psi(e^s)\).
The differential equation~\eqref{eq:fm-ode} with \(m=1\) gives
\begin{equation}\label{eq:psi-identity}
 z\psi'(z)=1-f(z)+zf'(z).
\end{equation}
Differentiating that equation gives
\[
 z(z-1)f''(z)+3zf'(z)+f(z)=0,
\]
and hence
\[
 Q''(\log z)
 =zf(z)+3z^2f'(z)+z^3f''(z)
 =z^2f''(z)\ge0.
\]
Thus \(Q\) is convex, and its supporting-line inequality gives
\begin{equation}\label{eq:scaled-psi}
 \psi(z/A)-\psi(z)\ge-(\log A)z\psi'(z).
\end{equation}

For \(d>0\), set \(z=(\log A)d/\theta\) and \(\ell=\log A\).
The left side of~\eqref{eq:scaled-shearer} is
\[
 \theta+(1-\theta)f\left(\frac z{1-\theta}\right)-f(z)
 +\frac\theta\ell\bigl(\psi(z/A)-\psi(z)\bigr).
\]
By~\eqref{eq:scaled-psi} and~\eqref{eq:psi-identity}, this is at least
\[
 (1-\theta)
 \left[f\left(\frac z{1-\theta}\right)-f(z)\right]
 -\theta zf'(z),
\]
which is nonnegative because convexity gives
\[
 f\left(\frac z{1-\theta}\right)-f(z)
 \ge\frac{\theta z}{1-\theta}f'(z).
\]
Equality holds at \(d=0\).

For \(r>1\), a stationary point of \(\theta\mapsto\log(r/(1-\theta))/\theta\) satisfies
\[
 \log\frac r{1-\theta}=\frac\theta{1-\theta}.
\]
Choosing \(K=1/(1-\theta)\) gives~\eqref{eq:K-equation}, and endpoint behavior shows that this is the unique minimum.
The function \(K-1-\log(rK)\) is strictly increasing for \(K>1\), so the defining root is unique.
At the minimizer, the minimum value is \(K(r)\).
For \(r=1\), the inequality \(\log K\le K-1\), with equality only at \(K=1\), first identifies \(K(1)=1\).
The inequality \(-\log(1-\theta)\ge\theta\) and the limit at zero then prove~\eqref{eq:r-optimum}.
Implicit differentiation gives
\[
 K'(r)=\frac{K(r)}{r(K(r)-1)}>0
 \qquad(r>1).
\]
Finally, rewriting~\eqref{eq:K-equation} gives
\[
 K(r)=-W_{-1}\left(-\frac1{er}\right),
\]
where \(W_{-1}\) is the lower real branch of the Lambert \(W\)-function. 
This is the definition given in~\cite{DKPS20Structure}.
The standard expansion of this branch at zero~\cite{CGH+96} gives~\eqref{eq:K-asymptotic}.
\end{proof}

\noindent
Using this analytic lemma, we record the key inequality used in the induction step.

\begin{lemma}
\label{lem:r-continuation}
Recall \(A\) and \(g\) from~\eqref{eq:r-local-parameters}.
For every nonempty graph \(G=(V,E)\) whose neighborhoods are fractionally \(r\)-colorable and every normalized positive weighting \(w\), the random experiment above satisfies for $v\in V$
\begin{equation}\label{eq:r-continuation}
 \EE Y_v\ge g(d_{G,w}(v)).
\end{equation}
\end{lemma}

\begin{proof}
Fix a target vertex \(v\in V(G)\).
Abbreviate \(d_v=d_{G,w}(v)\), and, whenever \(v\in V(G')\), write \(d'_v=d_{G',w'}(v)\).
Write \(\EE_x\) for expectation conditional on the pivot \(x\).
We first prove the three bounds
\begin{align}
 x=v:\quad
 &\EE_xY_v
 \ge\theta+(1-\theta)
 g\left(\frac{d_v}{1-\theta}\right),
 \label{eq:r-own-response}\\
 x\in N_G(v):\quad
 &\EE_xY_v
 \ge\frac1A g(d_v/A),
 \label{eq:r-neighbor-response}\\
 x\notin N_G[v]:\quad
 &\EE_xY_v\ge g(d_v).
 \label{eq:r-nonneighbor-response}
\end{align}
For the first inequality, the selection branch puts \(v\) in \(S\).
In the certificate branch \(v\) remains, and its expected new degree is
\[
 \EE_x[d'_v\mid\text{certificate}]
 =\frac Ar d_v=\frac{d_v}{1-\theta}.
\]
Jensen's inequality applies because \(g\) is convex.

For the second inequality, \(v\) survives precisely on the certificate branch when \(v\in J_x\), an event of probability \((1-\theta)/r=1/A\).
Conditional on this event, since \(J_x\) is an independent set, it excludes every common neighbor of \(x\) and \(v\).
If
\[
 c_{x,v}=\frac{w(N_G(x)\cap N_G(v))}{w(v)},
\]
then \(d'_v=(d_v-c_{x,v})/A\), and monotonicity of \(g\) proves~\eqref{eq:r-neighbor-response}.

Finally, suppose that \(x\notin N_G[v]\) and write \(B=N_G(x)\cap N_G(v)\).
Relative to the original degree, both branches remove the mass of \(B\), while the certificate branch retains the vertices of \(B\cap J_x\) with mass multiplied by \(A\).
Hence
\[
 \EE_xd'_v
 =d_v-\frac{w(B)}{w(v)}
 +(1-\theta)\frac Ar\frac{w(B)}{w(v)}
 =d_v.
\]
Another application of Jensen proves \eqref{eq:r-nonneighbor-response}.

Multiplying the three established bounds by \(w(x)\) and summing over the possible pivots \(x\) gives
\begin{align}
 \EE Y_v
 \ge g(d_v)+w(v)\bigg[
 &\theta+(1-\theta)g\left(\frac{d_v}{1-\theta}\right)
 -(1+d_v)g(d_v)
 +\frac{d_v}{A}g(d_v/A)\bigg]
 \ge g(d_v),
 \label{eq:r-ideal-branch}
\end{align}
where the final inequality is by \cref{lem:rColorableAnalytic}.
\end{proof}

\noindent
The full induction now follows easily.

\begin{proposition}
\label{prop:r-parametrized}
Fix \(0<\theta<1\), and take \(A\) and \(g\) from \eqref{eq:r-local-parameters}.
If every neighborhood of \(G\) is fractionally \(r\)-colorable, then, for every \(w\in\RRpos^{V(G)}\), there is a distribution \(\mu\) on \(\cI(G)\) such that for every $v\in V(G)$,
\begin{equation}\label{eq:r-parametrized}
 \mu(v)\ge
 f\left(\frac{\log A}{\theta}d_{G,w}(v)\right).
\end{equation}
\end{proposition}

\begin{proof}
We prove the proposition by induction on \(|V(G)|\).
The empty graph is immediate.
The neighborhood hypothesis is hereditary, since every neighborhood in an induced subgraph is an induced subgraph of a neighborhood in the original graph.

Now fix a nonempty \(G\) and suppose the result holds on every proper induced subgraph \(G'\) of \(G\).
For every normalized positive weighting \(w\), use the random experiment from the start of this subsection.
\Cref{lem:r-continuation} gives \(\EE Y_v\ge g(d_{G,w}(v))\) simultaneously for every target, and by construction \(\PP(G'=G)\le1-\theta\).
\Cref{lem:LDInductionDriver}, with \(h=0\) and \(\rho=1-\theta\), therefore gives $\delta(G)\le 0$.
Scale invariance gives the assertion for every positive weighting.
\end{proof}

\begin{proof}[Proof of \cref{thm:LDRColNbdhs}]
Fix \(G\) and \(w\) as in the theorem.
If \(r>1\), take \(\theta=\theta_*\) in \cref{prop:r-parametrized}.
\Cref{lem:rColorableAnalytic} gives \((\log A)/\theta_*=K(r)\), proving~\eqref{eq:LDRColMarginal}.

If \(r=1\), choose a sequence \(\theta_j\downarrow0\).
Then \(-\log(1-\theta_j)/\theta_j\to1=K(1)\).
For the fixed pair \((G,w)\), choose a convergent subsequence of the corresponding laws in the probability simplex on \(\cI(G)\).
By continuity of \(f\), the limiting law satisfies~\eqref{eq:LDRColMarginal}.
\end{proof}

\begin{proof}[Proof of \cref{thm:LDAsymptoticRColNbdhs}]
For fixed \(r\), equation~\eqref{eq:fm-asymptotic} with \(m=1\) gives
\[
 f(K(r)d)
 =(1-o(1))\frac{\log d}{K(r)d}.
\]
The result follows from \cref{thm:LDRColNbdhs}.
\end{proof}

\section{Independence polynomial bounds}
\label{sec:partition}

We formulate the main proof idea in the multivariate setting. 
The multivariate independence polynomial of a graph \(G=(V,E)\) is
\[ Z_G(\lambda) = \sum_{I\in \cI(G)}\prod_{v\in I}\lambda_v, \]
where $\lambda$ is a vector of fugacities indexed by $v$.

\subsection{Proof overview}

The proofs use the identity
\begin{equation}\label{eq:intro-deletion}
 Z_G(\lambda)
 =Z_{G-v}(\lambda)+\lambda_v Z_{G-N_G[v]}(\lambda).
\end{equation}
We seek a pivot vertex $v$ for which the inductive lower bounds for the two terms on the right combine to prove the desired estimate.
For $p\in[0,1]$, let
\[
 h(p)=-p\log p-(1-p)\log(1-p)
\]
be the base-$e$ Bernoulli entropy function (with $h(0)=h(1)=0$).
A standard variational identity (see \cref{prop:variational} in \cref{app:details}) rewrites \eqref{eq:intro-deletion} as
\begin{align}\label{eq:intro-log-sum}
 \log Z_G(\lambda)=\max_{p\in[0,1]}\biggl\{
  h(p)&+(1-p)\log Z_{G-v}(\lambda)
  +p\log\lambda_v+
       p\log Z_{G-N_G[v]}(\lambda)
 \biggr\}.
\end{align}
This is a convenient tool for analyzing random choices of the pivot vertex \(v\).

\subsection{General proof structure}

We prove the independence polynomial bounds directly from the identity~\eqref{eq:intro-deletion} by induction.
Fugacity vectors throughout this section are taken in the positive orthant. 
Throughout, we write $\ell=\log(1+\lambda)$, or $\ell_v=\log(1+\lambda_v)$ in the multivariate setting.

For \(\lambda>0\) and \(s\ge0\), set
\begin{equation}\label{eq:partition-F-definition}
 F_\lambda(s)
 =h(1-e^{-s})+(1-e^{-s})\log\lambda.
\end{equation}
The same elementary inequality drives both inductions.

\begin{lemma}\label{lem:partition-support}
Let \(\lambda>0\), write \(\ell=\log(1+\lambda)\), and let \(s\ge0\).
Then
\begin{equation}\label{eq:partition-support}
 F_\lambda(s)
 \ge s\left(1+\log\frac{\ell}{s}\right),
\end{equation}
where the right-hand side is interpreted as zero when \(s=0\).
\end{lemma}

\begin{proof}
The case \(s=0\) is immediate, so suppose that \(s>0\).
Let \(p=1-e^{-s}\) and \(p_*=\lambda/(1+\lambda)=1-e^{-\ell}\).
Let \(X\) and \(X_*\) have Poisson distributions with means \(s\) and \(\ell\), respectively.
The probabilities that they are positive are \(p\) and \(p_*\).
If all positive outcomes are merged into a single outcome, the log-sum inequality gives
\[
 \sum_{k\ge1}\PP(X=k)\log\frac{\PP(X=k)}{\PP(X_*=k)}
 \ge p\log\frac{p}{p_*}.
\]
Adding the contributions from the zero outcomes therefore gives
\[
 \dkl{\operatorname{Bern}(p)}{\operatorname{Bern}(p_*)}
 \le
 \dkl{\operatorname{Poi}(s)}{\operatorname{Poi}(\ell)}.
\]
The two sides are, respectively,
\[
 \ell-F_\lambda(s)
 \quad\text{and}\quad
 \ell-s+s\log\frac{s}{\ell}.
\]
Cancelling \(\ell\) and rearranging proves~\eqref{eq:partition-support}.
\end{proof}

\subsection{Sparse neighborhoods}

We prove a multivariate generalization of \cref{thm:IPSparseNbhdsDegrees}.

\begin{theorem}\label{thm:IPSparseNbhdsMulti}
Fix \(a\ge0\), let \(G=(V,E)\) be a graph with $V\ne\emptyset$ satisfying \(\mad(G[N(v)])\le a\) for every \(v\in V\), and let \(\lambda\in\RRpos^V\).
Write
\[
 \ell_v=\log(1+\lambda_v),\qquad
 L=\max_{v\in V}\ell_v,\qquad
 C=e^{aL}.
\]
Then
\begin{equation}\label{eq:sparse-partition}
 \log Z_G(\lambda)
 \ge
 \max_{w\in\RRpos^V}
 \left\{
  \sum_{v\in V}w(v)\left(1+\log\frac{\ell_v}{w(v)}\right)
  -C\sum_{uv\in E}w(u)w(v)
 \right\}.
\end{equation}
Consequently,
\begin{equation}\label{eq:sparse-partition-degrees}
 \log Z_G(\lambda)
 \ge
 \sum_{v\in V}
 \frac{\mathcal A\bigl(W(Cd(v)\ell_v)\bigr)}
      {Cd(v)},
\end{equation}
where the summand for \(d(v)=0\) is interpreted as \(\ell_v\).
\end{theorem}

The following lemma shows that, after choosing a maximum-weight pivot, the gain from its incident edges dominates the cost from edges inside its neighborhood.

\begin{lemma}\label{lem:sparse-partition-scalar}
Let \(a\ge1\), \(L>0\), \(C=e^{aL}\), and \(0<m\le L\).
Then
\begin{equation}\label{eq:sparse-partition-scalar}
 \frac{C(e^m-1)}2
 \min\left\{am,\frac1C\log\frac Lm\right\}
 \le (C-1)m.
\end{equation}
\end{lemma}

\begin{proof}
Write \(x=aL\) and \(t=m/L\).
Then \(C=e^x\), \(0<t\le1\), and \(m\le L\le x\).
The function
\[
 u\longmapsto\frac{e^u-1}{u}
\]
is increasing on \((0,\infty)\).
After division by the right-hand side of~\eqref{eq:sparse-partition-scalar}, the ratios corresponding to the two expressions in the minimum are bounded by
\[
 \frac{a(e^m-1)}{2(1-e^{-x})}
 \le \frac{te^x}{2}
 \qquad\text{and}\qquad
 \frac{(e^m-1)\log(1/t)}{2(C-1)m}
 \le \frac{\log(1/t)}{2x},
\]
respectively.
If \(\log(1/t)\le x\), the second ratio is at most \(1/2\).
Otherwise the first ratio is \(e^{x-\log(1/t)}/2<1/2\).
This proves the lemma.
\end{proof}

\begin{proof}[Proof of \cref{thm:IPSparseNbhdsMulti}]
Fix \(G,\lambda,L,C\) as in the theorem.
For every induced subgraph \(H\) of \(G\) and \(w\in\RRpos^{V(H)}\), define the potential
\[
 \Phi_H(w)=
 \sum_{x\in V(H)}
 w(x)\left(1+\log\frac{\ell_x}{w(x)}\right)
 -C\sum_{xy\in E(H)}w(x)w(y).
\]
The definition extends continuously to nonnegative \(w\) by giving a zero coordinate term the value zero.

We prove by induction on \(|V(H)|\) that
\begin{equation}\label{eq:sparse-partition-induction}
 \log Z_H(\lambda)\ge\max_{w\in\RRpos^{V(H)}}\Phi_H(w).
\end{equation}
The base case of the empty graph is immediate.

Suppose that \(H\) is nonempty.
The potential $\Phi_H(w)$ tends to \(-\infty\) when any coordinate tends to infinity, so it has a maximum.
No maximizing coordinate is zero: increasing such a coordinate from zero has a gain of order \(\varepsilon\log(1/\varepsilon)\) from the first sum and an edge cost of order \(\varepsilon\) from the second.
Thus a maximizer \(w\) is positive and satisfies for every $x\in V(H)$,
\begin{equation}\label{eq:sparse-partition-stationarity}
 \log\frac{\ell_x}{w(x)}
 =C\sum_{y\in N_H(x)}w(y).
\end{equation}

Choose \(v\) so that \(m=w(v)\) is maximal, and write
\[
 R=\sum_{xy\in E(H[N_H(v)])}w(x)w(y),\qquad S=w(N_H(v)).
\]
In particular,~\eqref{eq:sparse-partition-stationarity} gives
\begin{equation}\label{eq:sparse-partition-mS}
 m\le\ell_v\le L
 \qquad\text{and}\qquad
 S=\frac1C\log\frac{\ell_v}{m}
 \le\frac1C\log\frac Lm.
\end{equation}

Set \(p=1-e^{-m}\).
On \(H-v\), use the weights
\[
 w^v(x)=
 \begin{cases}
  e^m w(x),&x\in N_H(v),\\
  w(x),&x\notin N_H[v],
 \end{cases}
\]
and on \(H-N_H[v]\), use the restriction of \(w\).
The induction hypothesis and the log-sum identity \eqref{eq:intro-log-sum} yield
\begin{align}
 \log Z_H(\lambda)
 \ge{}&F_{\lambda_v}(m)
 +(1-p)\Phi_{H-v}(w^v)
 +p\Phi_{H-N_H[v]}(w).
 \label{eq:sparse-partition-step}
\end{align}
For \(x\in N_H(v)\),
\[
 (1-p)w^v(x)
 \left(1+\log\frac{\ell_x}{w^v(x)}\right)
 =
 w(x)\left(1+\log\frac{\ell_x}{w(x)}\right)-mw(x).
\]
Moreover, partitioning the edges according to their relation to \(N_H[v]\) gives
\begin{align*}
 &(1-p)\sum_{xy\in E(H-v)}w^v(x)w^v(y)
 +p\sum_{xy\in E(H-N_H[v])}w(x)w(y)\\
 &\qquad-\sum_{xy\in E(H)}w(x)w(y)
 =-mS+(e^m-1)R.
\end{align*}
It follows from \cref{lem:partition-support} that the right-hand side of~\eqref{eq:sparse-partition-step} minus \(\Phi_H(w)\) is at least
\begin{equation}\label{eq:sparse-partition-surplus}
 B=(C-1)mS-C(e^m-1)R.
\end{equation}

If \(R=0\), then \(B\ge0\).
Otherwise \(H[N_H(v)]\) contains an edge, so \(a\ge1\).
Since every weight in \(N_H(v)\) is at most \(m\), the weighted maximum-average-degree inequality \eqref{eq:vector-mad} (developed in the proof of \cref{prop:Theta}) gives
\[
 R
 \le m\sum_{xy\in E(H[N_H(v)])}\min\{w(x),w(y)\}
 \le\frac{amS}{2}.
\]
We also have
\[
 2R
 =\sum_{x\in N_H(v)}w(x)
   \sum_{y\in N_H(v)\cap N_H(x)}w(y)
 \le S^2,
\]
and hence
\[
 R\le\frac{S^2}{2}
 \le\frac{S}{2C}\log\frac Lm,
\]
where the second inequality is~\eqref{eq:sparse-partition-mS}.
Consequently,
\[
 R\le\frac S2
 \min\left\{am,\frac1C\log\frac Lm\right\}.
\]
\Cref{lem:sparse-partition-scalar} now shows that \(B\ge0\).
Thus~\eqref{eq:sparse-partition-step} is at least \(\Phi_H(w)\), which proves~\eqref{eq:sparse-partition-induction}.

Taking \(H=G\) proves~\eqref{eq:sparse-partition}.
Finally,
\[
 2w(u)w(v)\le w(u)^2+w(v)^2
\]
for every edge \(uv\), and hence the right-hand side of \eqref{eq:sparse-partition} is at least
\[
 \max_{w\in\RRpos^V}
 \sum_{v\in V}\left[
  w(v)\left(1+\log\frac{\ell_v}{w(v)}\right)
  -\frac{Cd(v)}2w(v)^2
 \right].
\]
This optimization separates over the vertices.
If \(d(v)>0\), its maximizer satisfies
\[
 \log\frac{\ell_v}{w(v)}=Cd(v)w(v),
 \qquad
 w(v)=\frac{W(Cd(v)\ell_v)}{Cd(v)},
\]
and substitution gives the corresponding summand in \eqref{eq:sparse-partition-degrees}.
For \(d(v)=0\), the maximum is attained at \(w(v)=\ell_v\) and has value \(\ell_v\).
\end{proof}

\begin{proof}[Proof of \cref{thm:IPSparseNbhdsDegrees}]
The assertion is immediate for the empty graph.
Otherwise apply \cref{thm:IPSparseNbhdsMulti} with \(\lambda_v=\lambda\) for every \(v\).
Then
\[
 L=\log(1+\lambda),
 \qquad
 C=e^{aL}=(1+\lambda)^a,
\]
and~\eqref{eq:sparse-partition-degrees} is exactly the claimed bound.
\end{proof}

\subsubsection{A more precise bound}

For $a,x,z\ge0$, define
\begin{equation}\label{eq:mad-bvdhk-profile}
 \mathcal B_{z,a}(x)=
 \begin{cases}
 \displaystyle
 \frac{\mathcal A(W(zx))-\mathcal A(W(z(a+2)))}
      {x-(a+2)},&x\ne a+2,\\[3mm]
 \displaystyle\frac{W(z(a+2))}{a+2},&x=a+2.
 \end{cases}
\end{equation}
The second line is the removable singularity in the first.
Thus, the right side of~\eqref{eq:IPSparseNbhdsBvdHKBound} is $\mathcal B_{\lambda,a}(d)$.

\begin{proof}[Proof of \cref{thm:IPSparseNbhdsBvdHK}]
The result is immediate when $\lambda=0$, so assume $\lambda>0$.
Write $c=a+2$ and define
\[
 P(x)=\mathcal A(W(\lambda x)),
 \qquad
 F(x)=\mathcal B_{\lambda,a}(x).
\]
We first establish the analytic inequality needed by the induction.
Direct differentiation gives, for $x>0$,
\begin{align}
 P'(x)&=\frac{W(\lambda x)}x
       =\lambda e^{-W(\lambda x)},\notag\\
 P''(x)&=-\frac{W(\lambda x)^2}
 {x^2(1+W(\lambda x))}\le0,\notag\\
 P'''(x)&=
 \frac{W(\lambda x)^3(2W(\lambda x)+3)}
 {x^3(1+W(\lambda x))^3}\ge0.
 \label{eq:mad-bvdhk-P-derivatives}
\end{align}
These derivatives extend continuously to zero, giving $P'(0)=\lambda$.
We also have
\begin{equation}\label{eq:mad-bvdhk-P''}
 P''(x)=-\frac{P'(x)^2}{1+x P'(x)}.
\end{equation}
Moreover,
\begin{equation}\label{eq:mad-bvdhk-divided-difference}
 F(x)=\int_0^1P'\bigl(c+s(x-c)\bigr)\,\mathrm ds.
\end{equation}
It follows that $F$ is positive, nonincreasing, and convex on $[0,\infty)$.

We claim that, for every $x>0$,
\begin{equation}\label{eq:mad-bvdhk-criterion}
 \exp\bigl(-xF'(x)-F(x)\bigr)
 +\lambda
 \exp\bigl(x(a+1-x)F'(x)-(x+1)F(x)\bigr)
 \ge1.
\end{equation}
Set $u=W(\lambda x)=xP'(x)$.
Differentiating $(x-c)F(x)=P(x)-P(c)$ gives
\begin{equation}\label{eq:mad-bvdhk-W-identity}
 xF(x)+x(x-c)F'(x)=u.
\end{equation}
If
\[
 \beta_0=-xF'(x)-F(x),
 \qquad
 \beta_1=x(c-1-x)F'(x)-(x+1)F(x),
\]
then
\begin{equation}\label{eq:mad-bvdhk-beta-relations}
 \beta_1-\beta_0=-u,
 \qquad
 \beta_0=\frac{cF(x)-xP'(x)}{x-c},
\end{equation}
with the quotient understood continuously at $x=c$.
Hence \eqref{eq:mad-bvdhk-criterion} is equivalent to
\begin{equation}\label{eq:mad-bvdhk-log-target}
 Q(x):=\beta_0+\log(1+P'(x))\ge0.
\end{equation}

To verify this, use the elementary inequality
$\log(1+z)\ge2z/(2+z)$ for $z\ge0$ and define
\begin{equation}\label{eq:mad-bvdhk-T}
 T(x)=c\bigl(P(x)-P(c)\bigr)-xP'(x)(x-c)
      +\frac{2P'(x)(x-c)^2}{2+P'(x)}.
\end{equation}
For $x\ne c$, equations~\eqref{eq:mad-bvdhk-beta-relations}
and~\eqref{eq:mad-bvdhk-T} give
\begin{equation}\label{eq:mad-bvdhk-T-lower}
 Q(x)
 \ge \beta_0+\frac{2P'(x)}{2+P'(x)}
 =\frac{T(x)}{(x-c)^2}.
\end{equation}
Using~\eqref{eq:mad-bvdhk-P''}, direct differentiation gives
\begin{equation}\label{eq:mad-bvdhk-T-prime}
 T'(x)=
 \frac{(x-c)P'(x)^2}{(2+P'(x))^2(1+xP'(x))}R(x),
 \qquad
 R(x)=4(c-1)-2P'(x)-xP'(x)^2.
\end{equation}
Another application of~\eqref{eq:mad-bvdhk-P''} gives
\[
 R'(x)=P'(x)^2,
 \qquad
 R(0)=4(c-1)-2\lambda\ge0,
\]
where the last inequality uses $\lambda\le2(c-1)=2(a+1)$.
Thus $R$ is nonnegative, so~\eqref{eq:mad-bvdhk-T-prime}
shows that $T$ is nonincreasing on $[0,c]$ and nondecreasing on
$[c,\infty)$. Since $T(c)=0$, we have $T(x)\ge0$ for all $x\ge0$.
Together with~\eqref{eq:mad-bvdhk-T-lower}, this proves
\eqref{eq:mad-bvdhk-log-target}; the point $x=c$ follows by continuity.
Letting $x\downarrow0$ in \eqref{eq:mad-bvdhk-criterion} also gives
\begin{equation}\label{eq:mad-bvdhk-endpoint}
 F(0)\le\log(1+\lambda).
\end{equation}

We now prove the theorem by induction.
Every induced subgraph $H$ of $G$ satisfies
\begin{equation}\label{eq:mad-hereditary-triangle-bound}
 3|\mathcal T(H)|
 =\sum_{v\in V(H)}|E(H[N_H(v)])|
 \le\frac a2\sum_{v\in V(H)}d_H(v)
 =a|E(H)|.
\end{equation}
Indeed, $H[N_H(v)]$ is a subgraph of $G[N_G(v)]$, so the maximum average degree hypothesis applies term by term.
Define
\[
 \psi(H)=
 \begin{cases}
  |V(H)|F(d(H)),&V(H)\ne\emptyset,\\
  0,&V(H)=\emptyset,
 \end{cases}
\]
where $d(H)$ is the average degree of $H$.
We first show that every nonempty $H$ contains a vertex $v$ such that
\begin{equation}\label{eq:mad-bvdhk-bellman-recurrence}
 e^{\psi(H)}
 \le e^{\psi(H-v)}+\lambda e^{\psi(H-N_H[v])}.
\end{equation}
If $H$ is edgeless, this follows from \eqref{eq:mad-bvdhk-endpoint}: for every $v\in V(H)$,
\[
 e^{\psi(H)}
 =e^{|V(H)|F(0)}
 \le(1+\lambda)e^{(|V(H)|-1)F(0)}
 =e^{\psi(H-v)}+\lambda e^{\psi(H-N_H[v])}.
\]

Otherwise, let $n=|V(H)|$, $m=|E(H)|$, and $d=2m/n>0$.
Set
\[
 H_v^{\mathrm{out}}=H-v,
 \qquad
 H_v^{\mathrm{in}}=H-N_H[v],
\]
and, for $\sigma\in\{\mathrm{out},\mathrm{in}\}$, write $n_v^\sigma=|V(H_v^\sigma)|$ and $m_v^\sigma=|E(H_v^\sigma)|$.
Define
\[
 X_v^\sigma=
 (n-n_v^\sigma)(dF'(d)-F(d))
 -2(m-m_v^\sigma)F'(d).
\]
Because \(F\) is convex, the function
\[
 (n',m')\longmapsto n'F\left(\frac{2m'}{n'}\right)
 \qquad(n'>0)
\]
is convex.
Its tangent inequality gives
\begin{equation}\label{eq:mad-bvdhk-tangent}
 \psi(H_v^\sigma)-\psi(H)\ge X_v^\sigma.
\end{equation}
Indeed, the gradient of the aforementioned function at \((n,m)\) is
\[
 \bigl(F(d)-dF'(d),\,2F'(d)\bigr),
\]
which gives exactly \(X_v^\sigma\).
If the child graph is empty, then \(n_v^\sigma=m_v^\sigma=0\) and \(X_v^\sigma=-\psi(H)\), so the same inequality holds with equality.

Choose a vertex $\boldsymbol v$ uniformly from $V(H)$.
Elementary deletion averages give
\begin{equation}\label{eq:mad-bvdhk-deletion-averages}
\begin{aligned}
 \EE(n-n_{\boldsymbol v}^{\mathrm{out}})&=1,
 &
 \EE(m-m_{\boldsymbol v}^{\mathrm{out}})&=d,\\
 \EE(n-n_{\boldsymbol v}^{\mathrm{in}})&=d+1,
 &
 \EE(m-m_{\boldsymbol v}^{\mathrm{in}})
 &=\frac1n\left(\sum_{u\in V(H)}d_H(u)^2
       -3|\mathcal T(H)|\right)\\
 & & &\ge d^2-\frac{ad}{2}.
\end{aligned}
\end{equation}
For the identity in the last line, an edge $yz$ is removed from the graph for $|N_H[y]\cup N_H[z]|$ choices of the pivot.
Since \(yz\in E(H)\),
\[
 |N_H[y]\cup N_H[z]|
 =d_H(y)+d_H(z)-|N_H(y)\cap N_H(z)|.
\]
Summing the first two terms over the edges gives \(\sum_{u\in V(H)}d_H(u)^2\), while the common-neighbor term counts every triangle three times.
The inequality follows from Cauchy--Schwarz and \eqref{eq:mad-hereditary-triangle-bound}.
Note that in the case $a=0$ these arguments are exactly those of Buys, van den Heuvel, and Kang~\cite{BHK25}, and we observe an extra term arising from triangles in~\eqref{eq:mad-bvdhk-deletion-averages}.

Since $F'(d)\le0$, equations~\eqref{eq:mad-bvdhk-tangent} and \eqref{eq:mad-bvdhk-deletion-averages} imply
\begin{align*}
 \EE X_{\boldsymbol v}^{\mathrm{out}}
 &=-dF'(d)-F(d),\\
 \EE X_{\boldsymbol v}^{\mathrm{in}}
 &\ge d(a+1-d)F'(d)-(d+1)F(d).
\end{align*}
Jensen's inequality and~\eqref{eq:mad-bvdhk-criterion} now give
\[
 \EE\left(e^{X_{\boldsymbol v}^{\mathrm{out}}}
       +\lambda e^{X_{\boldsymbol v}^{\mathrm{in}}}\right)\ge1.
\]
Some vertex makes the expression in parentheses at least one, and \eqref{eq:mad-bvdhk-tangent} proves \eqref{eq:mad-bvdhk-bellman-recurrence}.

Finally, we prove $Z_H(\lambda)\ge e^{\psi(H)}$ by induction on $|V(H)|$.
It is immediate for the empty graph.
For nonempty $H$, choose $v$ satisfying~\eqref{eq:mad-bvdhk-bellman-recurrence}.
The induction hypothesis and the deletion identity~\eqref{eq:intro-deletion} give
\[
 Z_H(\lambda)
 =Z_{H-v}(\lambda)+\lambda Z_{H-N_H[v]}(\lambda)
 \ge e^{\psi(H-v)}+\lambda e^{\psi(H-N_H[v])}
 \ge e^{\psi(H)}.
\]
Taking $H=G$ gives $|V(G)|^{-1}\log Z_G(\lambda)\ge F(d)$, which is exactly \eqref{eq:IPSparseNbhdsBvdHKBound}.
\end{proof}

\subsection{Fractionally colorable neighborhoods}\label{subsec:IPRColNbdhs}

Here the potential used for \cref{thm:IPSparseNbhdsMulti} works with the constant $C$ replaced by \(K(r)\) from~\eqref{eq:K-equation}.

\begin{theorem}\label{thm:IPRColMulti}
Fix \(r\ge1\), let \(G=(V,E)\) have fractionally \(r\)-colorable neighborhoods, and let \(\lambda\in\RRpos^V\).
Write \(\ell_v=\log(1+\lambda_v)\).
Then
\begin{equation}\label{eq:r-partition}
 \log Z_G(\lambda)
 \ge
 \sup_{w\in\RRpos^V}
 \left\{
  \sum_{v\in V}w(v)\left(1+\log\frac{\ell_v}{w(v)}\right)
  -K(r)\sum_{uv\in E}w(u)w(v)
 \right\}.
\end{equation}
Consequently,
\begin{equation}\label{eq:r-partition-multidegrees}
 \log Z_G(\lambda)
 \ge
 \sum_{v\in V}
 \frac{\mathcal A\bigl(W(K(r)d(v)\ell_v)\bigr)}
      {K(r)d(v)},
\end{equation}
where the summand for \(d(v)=0\) is interpreted as \(\ell_v\).
\end{theorem}

\begin{proof}
Write \(K=K(r)\).
For an induced subgraph \(H\) of \(G\) and \(w\in\RRpos^{V(H)}\), define the potential
\[
 \Psi_H(w)=
 \sum_{x\in V(H)}
 w(x)\left(1+\log\frac{\ell_x}{w(x)}\right)
 -K\sum_{xy\in E(H)}w(x)w(y).
\]
We prove by induction that
\begin{equation}\label{eq:r-partition-induction}
 \log Z_H(\lambda)\ge\sup_{w\in\RRpos^{V(H)}}\Psi_H(w)
\end{equation}
for every induced subgraph \(H\).
The empty graph is immediate.
Suppose that \(H\) is nonempty, fix \(w\in\RRpos^{V(H)}\), and choose any pivot \(v\).
Set
\[
 m=w(v),\qquad
 p=1-e^{-m},\qquad
 \theta=\max\left\{1-\frac1K,p\right\},\qquad
 \alpha=\frac{p}{\theta},\qquad
 A=\frac{r}{1-\theta}.
\]
These parameters satisfy
\begin{equation}\label{eq:r-partition-parameters}
 0<\alpha\le1,\qquad
 \alpha\theta=p,\qquad
 \frac{(1-\theta)A}{r}=1,\qquad
 \alpha\log A\le Km.
\end{equation}
Only the last inequality needs verification.
If \(m\le\log K\), a case that can occur only when \(K>1\), then
\[
 \theta=1-\frac1K,\qquad
 \log A=\log(rK)=K-1,
\]
and hence
\[
 \alpha\log A=K(1-e^{-m})\le Km.
\]
If \(m\ge\log K\), then \(\theta=p\), \(\alpha=1\), and
\[
 \alpha\log A=m+\log r\le Km.
\]
For the last inequality, use
\[
 \log r=K-1-\log K\le(K-1)\log K\le(K-1)m;
\]
the middle inequality is equivalent to \(K\log K\ge K-1\).
This also covers \(K=1\).

Let \(J\) be a random independent set of \(H[N_H(v)]\) with \(\PP(x\in J)=1/r\) for every \(x\in N_H(v)\).
Define
\[
 H_J=H-\bigl(\{v\}\cup(N_H(v)\setminus J)\bigr),
\]
and give its vertices the weights
\[
 w^J(x)=
 \begin{cases}
  A w(x),&x\in J,\\
  w(x),&x\notin N_H[v].
 \end{cases}
\]
Every \(H_J\) is an induced subgraph of \(H-v\), and all fugacities are positive, so \(Z_{H-v}(\lambda)\ge Z_{H_J}(\lambda)\).
Since
\[
 (1-\alpha)+\alpha(1-\theta)=1-p,
\]
the induction hypothesis gives
\begin{align}
 (1-p)\log Z_{H-v}(\lambda)
 \ge{}&(1-\alpha)\Psi_{H-v}(w)\notag\\
 &+\alpha(1-\theta)\EE_J\Psi_{H_J}(w^J).
 \label{eq:r-partition-nonselection}
\end{align}
Combining this with the induction bound on \(H-N_H[v]\) and applying the log-sum identity~\eqref{eq:intro-log-sum} with selection weight \(p\), we obtain
\begin{align}
 \log Z_H(\lambda)
 \ge{}&F_{\lambda_v}(m)
 +(1-\alpha)\Psi_{H-v}(w)
 +p\Psi_{H-N_H[v]}(w)\notag\\
 &+\alpha(1-\theta)\EE_J\Psi_{H_J}(w^J).
 \label{eq:r-partition-step}
\end{align}

For brevity, write
\[
 S=w(N_H(v)),
 \qquad
 R=\sum_{xy\in E(H[N_H(v)])}w(x)w(y).
\]
For each \(x\in N_H(v)\), the expected vertex term in the three child potentials is
\begin{gather*}
 (1-\alpha)w(x)\left(1+\log\frac{\ell_x}{w(x)}\right)+\frac{\alpha(1-\theta)}r
  A w(x)\left(1+\log\frac{\ell_x}{A w(x)}\right) \\ =w(x)\left(1+\log\frac{\ell_x}{w(x)}\right)
   -\alpha w(x)\log A.
\end{gather*}
Edges incident to \(v\) disappear, producing a gain \(KmS\).
An edge with one endpoint in \(N_H(v)\) and the other outside \(N_H[v]\) has unchanged expected mass by \eqref{eq:r-partition-parameters}.
Since \(J\) is independent, an edge inside \(N_H(v)\) survives only in the plain child, producing a gain \(K\alpha R\).
It follows that the child-potential terms on the right of~\eqref{eq:r-partition-step}, minus \(\Psi_H(w)\), equal
\[
 -m\left(1+\log\frac{\ell_v}{m}\right)
 +(Km-\alpha\log A)S+K\alpha R.
\]
\Cref{lem:partition-support} cancels the first term, and the remaining terms are nonnegative by \eqref{eq:r-partition-parameters}.
Thus \(\log Z_H(\lambda)\ge\Psi_H(w)\).
Since \(w\) was arbitrary, \eqref{eq:r-partition-induction} follows, completing the induction and proving~\eqref{eq:r-partition}.

Finally, \(2w(u)w(v)\le w(u)^2+w(v)^2\) shows that the right-hand side of~\eqref{eq:r-partition} is at least
\[
 \sup_{w\in\RRpos^V}
 \sum_{v\in V}\left[
  w(v)\left(1+\log\frac{\ell_v}{w(v)}\right)
  -\frac{Kd(v)}2w(v)^2
 \right].
\]
For \(d(v)>0\), the maximizing coordinate is
\[
 w(v)=\frac{W(Kd(v)\ell_v)}{Kd(v)},
\]
and its value is \(\mathcal A(W(Kd(v)\ell_v))/(Kd(v))\).
For an isolated vertex the maximum is \(\ell_v\).
This proves \eqref{eq:r-partition-multidegrees}.
\end{proof}

\begin{proof}[Proof of \cref{thm:IPRColDegrees}]
Apply~\eqref{eq:r-partition-multidegrees} with \(\lambda_v=\lambda\) for every \(v\).
This gives \eqref{eq:IPRColDegreesBound}, including its continuous value at isolated vertices.
\end{proof}

When \(r=1\), the neighborhood hypothesis is exactly triangle-freeness and \(K(1)=1\).
Thus \cref{thm:IPRColMulti,thm:IPSparseNbhdsMulti} agree in the case \(a=0\).

\section{Tool and computational resource disclosure}

The author used GPT-5.6 Sol for literature search, explorations of ideas, and checking the manuscript. 
The author also pursued a formalization of the main results presented here in Lean 4 with Mathlib, in which Lean proofs were generated almost exclusively by AI.
This formalization will be made public shortly after the submission of this paper to a preprint server.

\appendix
\section{Deferred calculations and entropy analysis}\label[appendix]{app:details}

\begin{proposition}\label{prop:summands}
The summands in \cref{thm:IPSparseNbhdsDegrees,thm:IPRColDegrees} are decreasing and convex in \(d(v)\) for \(d(v)\ge0\).
\end{proposition}
\begin{proof}
Set \(\ell=\log(1+\lambda)\), and for \(c>0\) define
\[
 q_c(x)=\frac{\mathcal A(W(c\ell x))}{cx}
 \qquad(x>0).
\]
With \(u=W(c\ell x)\), we have \(ue^u=c\ell x\), and hence
\[
 q_c(x)=\ell\left(1+\frac u2\right)e^{-u},
 \qquad
 u'(x)=\frac{u}{x(1+u)}.
\]
It follows that
\[
 q_c'(x)=-\frac{c\ell^2}{2}e^{-2u}<0,
 \qquad
 q_c''(x)=c\ell^2e^{-2u}\frac{u}{x(1+u)}>0.
\]
As \(x\downarrow0\), we have \(u\downarrow0\) and \(u/x\to c\ell\), so
\[
 q_c(x)\longrightarrow\ell,
 \qquad
 q_c'(x)\longrightarrow-\frac{c\ell^2}{2},
 \qquad
 q_c''(x)\longrightarrow c^2\ell^3.
\]
Thus the continuous extension \(q_c(0)=\ell\) is decreasing and convex on \([0,\infty)\).
Taking \(c=(1+\lambda)^a\) or \(c=K(r)\) gives the result. 
\end{proof}

We record a fact used in the independence polynomial proofs.

\begin{proposition}\label{prop:variational}
For all \(a,b>0\),
\begin{equation}\label{eq:two-term-log-sum}
 \log(a+b)=\max_{p\in[0,1]}
 \left\{h(p)+(1-p)\log a+p\log b\right\}.
\end{equation}
The unique maximizer is \(p=b/(a+b)\).
In particular, taking \(a=Z_{G-v}(\lambda)\) and
\(b=\lambda_v Z_{G-N_G[v]}(\lambda)\) in~\eqref{eq:two-term-log-sum}
and using~\eqref{eq:intro-deletion} gives~\eqref{eq:intro-log-sum}.
\end{proposition}

\begin{proof}
Set
\[
 \phi(p)=h(p)+(1-p)\log a+p\log b.
\]
For \(p\in(0,1)\),
\[
 \phi'(p)=\log\frac{(1-p)b}{pa},
 \qquad
 \phi''(p)=-\frac1{p(1-p)}<0.
\]
Thus \(\phi\) is strictly concave, and its unique critical point
\(p=b/(a+b)\) is its unique maximum on \([0,1]\).
Substitution at this point gives \(\phi(p)=\log(a+b)\), proving the identity.
\end{proof}

The local demand theorems also give an elementary independence polynomial bound.
Fix \(\lambda>0\), set \(\ell=\log(1+\lambda)\), and let \(G\) be triangle-free.
Applying \cref{thm:LDRColNbdhs} with \(r=1\) and uniform vertex weights gives, by the probabilistic method, that some \(I_0\in\cI(G)\) satisfies
\[
 |I_0|\ge\sum_{v\in V}f(d(v)).
\]
Every subset of \(I_0\) is independent, so
\[
 \log Z_G(\lambda)
 \ge\ell\sum_{v\in V}f(d(v)).
\]

For a \(d\)-regular triangle-free graph, this gives
\[
 \log Z_G(\lambda)\ge |V|\ell f(d)
 =(1+o(1))\frac{|V|\ell\log d}{d},
\]
whereas \cref{thm:IPRColDegrees} gives exactly
\[
 \log Z_G(\lambda)
 \ge\frac{|V|}{d}\mathcal A(W(d\ell))
 =\left(\frac12+o(1)\right)
   \frac{|V|[\log(d\ell)]^2}{d}
\]
for fixed \(\lambda>0\) as \(d\to\infty\).
The ratio between these two lower-bound expressions is
\[
 (1+o(1))\frac{[\log(d\ell)]^2}{2\ell\log d}
 \sim\frac{\log d}{2\ell}.
\]
Thus the direct entropy-optimized induction gains a factor of order \(\log d\) over the elementary consequence of the local demand result.

\section{Random processes}\label[appendix]{app:random_processes}

A useful way to think about the local demand proofs is through a \emph{piecewise deterministic Markov process} (PDMP).
We give only an informal picture here; see~\cite{Davis84} for the general theory.
Moreover, we only highlight the process for the sparse neighborhoods result, leaving out the adaptations for fractionally colorable neighborhoods.
Given a graph $G=(V,E)$, a state $X=(I,R,w)$ consists of an independent set $I\in\cI(G)$, a set $R\subseteq V$ of pending vertices, and positive weights $w\in\RRpos^R$, with $N_G[I]\cap R=\emptyset$.
The vector $w$ plays two roles: $w(x)$ is the instantaneous clock rate of a pending vertex $x$, and the same weights define the local degrees in $G[R]$.
Write $X_t=(I_t,R_t,w_t)$ for the evolving state.
The starting state is $X_0 = (\emptyset,V,w_0)$, where $w_0$ is the initial weight vector.
In state $X$, when the clock at some $x$ rings, a prescribed response produces a possibly random post-response state $X^x$.
Between jumps, the rates follow an ordinary differential equation of the form
\begin{equation}\label{eq:pedagogical-weight-flow}
 \frac{\mathrm d}{\mathrm dt}\log w_t(x)=\Lambda_x(X_t)
 \qquad(x\in R_t),
\end{equation}
where the logarithmic drift $\Lambda_x$ is part of the construction.

Fix a decreasing demand function $g$, differentiable where needed below, and arbitrary nonnegative dual weights $u\in\RRnn^V$.
We consider the potential
\begin{equation}\label{eq:pedagogical-local-potential}
 \Phi_u(I,R,w)=
 \sum_{v\in I}u(v)
 +\sum_{v\in R}u(v)g\bigl(d_{G[R],w}(v)\bigr).
\end{equation}
At the initial state $(\emptyset,V,w_0)$ this is $\sum_{v\in V}u(v)g(d_{G,w_0}(v))$, while at an absorbing state, where $R=\emptyset$, it is the total dual weight $\sum_{v\in I_\infty}u(v)$.
The formal generator is
\begin{equation}\label{eq:pedagogical-generator}
 \mathcal L\Phi_u
 =\sum_{x\in R}w(x)
   \bigl(\EE[\Phi_u(X^x)\mid X]-\Phi_u(X)\bigr)
 +\sum_{x\in R}\frac{\mathrm d w_t(x)}{\mathrm dt}
   \frac{\partial\Phi_u}{\partial w(x)}.
\end{equation}
Thus, responses and a drift for which $\mathcal L\Phi_u\ge0$ make $\Phi_u(X_t)$ formally a submartingale.
At the absorbing time this would give
\[
 \EE\sum_{v\in I_\infty}u(v)
 \ge \sum_{v\in V}u(v)g(d_{G,w_0}(v)).
\]
If this can be arranged for every $u\in\RRnn^V$, then a standard linear-programming duality argument gives a distribution on independent sets with the desired marginals.

For triangle-free graphs with \(g=f=f_1\) and responses that simply select the vertex whose clock rang, one can \emph{derive} the correct logarithmic drift by analyzing the generator and using convexity in a manner analogous to Shearer's original proof~\cite{She83}.
The following calculation shows how the differential equation satisfied by \(f\) dictates the drift.

Write \(H=G[R]\), with all neighborhoods in the calculation taken in \(H\).
Take the response to a ring at \(x\) to select \(x\), remove \(N_H[x]\), and leave the weights of the surviving vertices unchanged.
By linearity, it suffices to show that the coefficient of \(u(v)\) in \(\mathcal L\Phi_u\) is nonnegative for every \(v\).
This coefficient is constant if \(v\) has already been selected or discarded, so fix \(v\in R\) and write
\[
 d_w(v)=d_{H,w}(v)=\frac{w(N_H(v))}{w(v)}.
\]
If \(d_w(v)=0\), then \(f(d_w(v))=1\), a ring at \(v\) leaves the coefficient equal to one, and no other clock or weight flow affects it.
Thus the coefficient has generator zero.
We may therefore assume that \(d_w(v)>0\).

For \(x\in R\setminus N_H[v]\), write
\[
 c_{x,v}=\frac{w(N_H(v)\cap N_H(x))}{w(v)}.
\]
A ring at \(x\) leaves \(v\) pending and decreases its weighted degree from \(d_w(v)\) to \(d_w(v)-c_{x,v}\).
A ring at \(v\) selects \(v\), while a ring in \(N_H(v)\) discards it.
Hence the jump contribution to the coefficient of \(u(v)\) is
\begin{align*}
 {}&w(v)\bigl[1-f(d_w(v))\bigr]-w(N_H(v))f(d_w(v))
 +\sum_{x\in R\setminus N_H[v]}w(x)
    \bigl[f(d_w(v)-c_{x,v})-f(d_w(v))\bigr]\\
 ={}&w(v)\bigl[1-(1+d_w(v))f(d_w(v))\bigr]
 +\sum_{x\in R\setminus N_H[v]}w(x)
    \bigl[f(d_w(v)-c_{x,v})-f(d_w(v))\bigr].
\end{align*}
Convexity gives
\[
 f(d_w(v)-c_{x,v})-f(d_w(v))
 \ge-c_{x,v}f'(d_w(v)).
\]
Moreover, since \(H\) is triangle-free, the only vertex of \(N_H(y)\) that lies in \(N_H[v]\) is \(v\), for every \(y\in N_H(v)\).
Reversing the order of summation therefore gives
\begin{align*}
 \sum_{x\in R\setminus N_H[v]}w(x)c_{x,v}
 &=\frac1{w(v)}
   \sum_{y\in N_H(v)}w(y)\bigl(w(N_H(y))-w(v)\bigr)\\
 &=\frac1{w(v)}
   \sum_{y\in N_H(v)}w(y)w(N_H(y))-w(N_H(v)).
\end{align*}
Consequently the jump contribution is at least
\begin{equation}\label{eq:pedagogical-jump-bound}
 w(v)\bigl[1-(1+d_w(v))f(d_w(v))\bigr]
 -f'(d_w(v))\left[
   \frac1{w(v)}\sum_{y\in N_H(v)}w(y)w(N_H(y))
   -w(N_H(v))
 \right].
\end{equation}

Apply the differential equation~\eqref{eq:fm-ode} with \(m=1\) immediately to~\eqref{eq:pedagogical-jump-bound}.
Since
\[
 1-(1+d_w(v))f(d_w(v))
 =d_w(v)(d_w(v)-1)f'(d_w(v))
 \qquad\text{and}\qquad
 w(N_H(v))=w(v)d_w(v),
\]
the jump contribution is at least
\[
 -f'(d_w(v))\left[
   \frac1{w(v)}\sum_{y\in N_H(v)}w(y)w(N_H(y))
   -d_w(v)w(N_H(v))
 \right],
\]
We choose the continuous flow so that its contribution cancels this lower-bound term.
For an arbitrary logarithmic drift \(\Lambda\), differentiation gives
\[
 \frac{\mathrm d}{\mathrm dt}d_w(v)
 =\frac1{w(v)}\sum_{y\in N_H(v)}w(y)\Lambda_y(X)
  -d_w(v)\Lambda_v(X).
\]
Thus the choice, for every $x\in R$, $\Lambda_x(X)=w(N_H(x))$
makes \(\frac{\mathrm d}{\mathrm dt}d_w(v)\) equal to the quantity in brackets in the jump bound.
The flow contribution is \(f'(d_w(v))\frac{\mathrm d}{\mathrm dt}d_w(v)\), so it exactly cancels the displayed lower-bound term.
The coefficient of every \(u(v)\) in \(\mathcal L\Phi_u\) is therefore nonnegative.
Since the dual weights are nonnegative, it follows that \(\mathcal L\Phi_u\ge0\).

For graphs with sparse neighborhoods, the same calculation produces additional error terms arising from triangles.
\Cref{lem:sparse-obstacle-data} uses the weighted Laplacian to construct a correction to the logarithmic drift whose contribution dominates those terms.

Turning the formal generator argument into a proof would also require a discussion of nonexplosion, integrability, and a justified passage to the absorbing time.
These issues add technical details but do not change the core idea.
They appear to be manageable by constructing the process step by step through an induction, although that approach appears to defeat the main advantage of the PDMP formulation.

A benefit of the PDMP perspective is that it provides a clear basis for an efficient sampling algorithm for the distributions in $\cI(G)$ guaranteed by the theorem. This connection to a randomized greedy algorithm was highlighted by Shearer in his original work~\cite{She83}, and we feel it interesting to attempt to preserve this aspect of the method.

\printbibliography

\end{document}